\documentclass[10pt, a4paper, reqno]{amsart}
\usepackage[foot]{amsaddr}
\usepackage[margin=2.5cm]{geometry}
\usepackage{graphicx} 
\usepackage{amsthm, amsmath, amsfonts, amssymb, amscd, enumerate, enumitem, esint, xcolor, url, mathrsfs}
\usepackage{hyperref}

\usepackage[utf8]{inputenc}
\usepackage[T1]{fontenc}
\usepackage{mathtools}
\usepackage{enumerate}
\usepackage{textcomp}

\newtheorem{thm}{Theorem}[section]
\newtheorem{cor}[thm]{Corollary}

\newtheorem{prop}[thm]{Proposition}
\theoremstyle{definition}

\theoremstyle{remark}
\newtheorem{rem}[thm]{Remark}

\DeclareMathOperator{\LLog}{\mathbb{L}\textup{og}}
\DeclareMathOperator{\Log}{\textup{Log}}
\DeclareMathOperator{\supp}{supp}

\newcommand{\RR}{\mathbb{R}}
\newcommand{\NN}{\mathbb{N}}

\newcommand{\RH}{\textup{RH}}
\newcommand{\Lip}{\textup{Lip}}

\numberwithin{equation}{section}
\usepackage[defaultlines=2, all]{nowidow}
\setnoclub[2]
\allowdisplaybreaks

\setlist[enumerate,1]{label=\normalfont{(\alph*)}}

\begin{document}

	\title[Logarithmic Schr\"odinger operators]
	{Logarithmic Schr\"odinger operators}

	\author[J. J. Betancor]{Jorge J. Betancor$^1$}
        \address{$^1$Departamento de An\'alisis Matem\'atico, Universidad de La Laguna,\newline
	Campus de Anchieta, Avda. Astrof\'isico S\'anchez, s/n,\newline
	38721 La Laguna (Sta. Cruz de Tenerife), Spain}
	\email[J. J. Betancor]{jbetanco@ull.es}
	
	\author[E. Dalmasso]{Estefan\'ia Dalmasso$^{2}$}
	\address{$^2$Instituto de Matem\'atica Aplicada del Litoral, UNL, CONICET,     FIQ.\newline 
    Colectora Ruta Nac. Nº 168, Paraje El Pozo,\newline S3007ABA, Santa Fe, Argentina}
	\email[E. Dalmasso]{edalmasso@santafe-conicet.gov.ar}

    \author[J. C. Fariña]{Juan Carlos Fariña$^1$}  
        \email[J. C. Fariña]{jcfarina@ull.edu.es}
	
	\author[P. Quijano]{Pablo Quijano$^{2}$}
	\email[P. Quijano]{pquijano@santafe-conicet.gov.ar}
	
	\date{\today}
	

\begin{abstract}
    In this paper we consider the Schr\"odinger operator $\mathcal L_V= -\Delta + V$ in $\mathbb R^d$ with a non negative potential $V$, and $V\not\equiv 0$. We define the logarithmic Schr\"odinger operator $\log \mathcal L_V$ proving its main properties. We obtain a pointwise representation of $\log \mathcal L_V$ when $V$ satisfies a reverse H\"older inequality of exponent $q> \frac{d}{2}$ by using the semigroup of operators $\{T_t^V\}_{t>0}$ generated by $\mathcal L_V$. We consider the Lipschitz function space adapted to the Schr\"odinger setting to solve the initial value problem
    \[\begin{dcases*}
    \frac{\partial u}{\partial t}=-(\log\mathcal L_V)u, & $\text{in } \mathbb R^n\times (0,\infty)$ \\
    u(x,0)=f(x), & $x\in \mathbb R^d$
    \end{dcases*}\]
      in terms of the fractional integral associated with $\mathcal L_V$.
\end{abstract}

\maketitle

\section{Introduction}
Assume that $V$ is a nonnegative potential in $L^1_{\textup{loc}}(\mathbb{R}^d)$. We consider the sesquilinear form $T$ defined by 
\begin{equation*}
    T(f,g) = \int_{\mathbb{R}^d} \nabla f \cdot \overline{\nabla g} \, dx + \int_{\mathbb{R}^d} V f \,\overline{g}\, dx,
\end{equation*}
where $f, g \in D_T = \{h \in L^2(\mathbb{R}^d): \nabla h \in L^2(\mathbb{R}^d) \text{ and } \sqrt{V} h \in L^2(\mathbb{R}^d)\}$. Since $D_T$ is dense in $L^2(\mathbb{R}^d)$ and the quadratic form defined by $T$ is closed and nonnegative, there exists a self-adjoint operator \linebreak $\mathcal{L}_V: D_T =  D(\mathcal{L}_V)  \rightarrow L^2(\mathbb{R}^d)$ such that $\langle \mathcal{L}_V f, g\rangle = T(f,g)$ for all $f, g \in D(\mathcal{L}_V)$ \cite[Theorem VIII.15]{ReSi}. Moreover, for every $f \in C^\infty_c(\mathbb{R}^d)$, the space of $C^\infty$ functions with compact support in $\mathbb{R}^d$, we have $\mathcal{L}_V f = -\Delta f + V f$.

Our objective is to define and study the logarithmic operator $\log \mathcal{L}_V$. This operator is defined using functional calculus. We prove that, alternatively, $\log \mathcal{L}_V$ can be defined as the first variation of the fractional operator $\mathcal{L}_V^{s}$ at $s=0$. We give a pointwise representation of the operator $\log \mathcal{L}_V$. We also solve the initial value problem $\frac{\partial u}{\partial t}=-(\log\mathcal L_V)u$ in $\mathbb R^n\times (0,\infty)$, and $u(x,0)=f(x)$, $x\in \mathbb R^d$, by using fractional integrals and Lipschitz spaces in the Schr\"odinger setting.

In the last decade, nonlocal operators have been studied by many authors. These operators appear in physics, probability theory, geometry, and applied mathematics (see, for instance,~\cite{CS,  FJW, JX, LL, LPW, Lu, RS, SV}). One of the most important examples of nonlocal operators is the fractional Laplacian $(-\Delta)^s$, $s>0$. The logarithmic Laplacian $\log (-\Delta)$ is a nonlocal operator that has received a lot of attention in recent years (see~\cite{Ch-orderm, ChHW, ChV2, ChV1, ChW, Ch-manifolds, ChX, Feu1, Feu2, HLS, HLW, LW, Lee}). The logarithmic Laplacian $\log (-\Delta)$ is closely connected with the fractional Laplacian $(-\Delta)^s$ by $\log (-\Delta) = \partial_s (-\Delta)^s |_{s=0}$. Recently, the study of $\log(-\Delta)$ has been extended to general Riemannian manifolds (\cite{Ch-manifolds}) and general graphs (\cite{ChX}) 

We now define the logarithmic Schr\"odinger operator $\log \mathcal{L}_V$, which is the object of our study. The spectrum $\sigma(\mathcal{L}_V)$ is contained in $[0,+\infty)$. Moreover, $0$ is not an eigenvalue of $\mathcal{L}_V$. Indeed, suppose that $f \in D(\mathcal{L}_V)$ is such that $\mathcal{L}_V f = 0$. Then,
\begin{equation*}
    \langle \mathcal{L}_V f, f \rangle = \int_{\mathbb{R}^d} |\nabla f|^2 \, dx + \int_{\mathbb{R}^d} V |f|^2 \, dx = 0.
\end{equation*}
Hence, $\nabla f = 0$ a.e. in $\mathbb{R}^d$, where the gradient is understood distributionally, and $f(x) = 0$ a.e. \linebreak $x \in \{y \in \mathbb{R}^d : V(y) > 0\}$. Since $f \in L^2(\mathbb{R}^d)$, then $f(x) = 0$ a.e. $x \in \mathbb{R}^d$.

There exists a spectral measure $E_V$ supported on $\sigma(\mathcal{L}_V)$ such that 
\begin{equation*}
    \mathcal{L}_V f = \int_{[0,+\infty)} \lambda \, dE_V(\lambda) f = \int_0^{\infty} \lambda \, dE_V(\lambda) f, \quad f \in D(\mathcal{L}_V), 
\end{equation*}
and 
\begin{equation*}
    D(\mathcal{L}_V) = \left\{ 
    f \in L^2(\mathbb{R}^d) : \int_{0}^\infty \lambda^2 \, d\mu_{f,f}^V(\lambda) < \infty \right\}.
\end{equation*}
Here, for every $f, g \in L^2(\mathbb{R}^d)$, $\mu_{f,g}^V$ is a complex Borel measure given by $\mu_{f,g}^V(U) = \langle E_V(U) f, g \rangle$, for every Borel measurable set $U \subseteq \mathbb{R}$. Note that $E_V(\{0\}) = 0$ because $0$ is not an eigenvalue of $\mathcal{L}_V$.

If $\phi$ is a Borel measurable function on $\mathbb{R}$, we define
\begin{equation*}
    \phi(\mathcal{L}_V) f = \int_0^\infty \phi(\lambda) \, dE_V(\lambda) f, \quad \text{for every } f \in D(\phi(\mathcal{L}_V)),
\end{equation*}
where 
\begin{equation*}
    D(\phi(\mathcal{L}_V)) = \left\{ 
    f \in L^2(\mathbb{R}^d) : \int_{0}^\infty |\phi(\lambda)|^2 \, d\mu_{f,f}^V(\lambda) < \infty \right\}.
\end{equation*}

The logarithmic operator $\log \mathcal{L}_V$ is therefore defined by
\[
(\log \mathcal{L}_V) f = \int_0^\infty \log(\lambda) \, dE_V(\lambda) f, \quad f \in D(\log \mathcal{L}_V),
\]
where
\begin{equation*}
    D(\log \mathcal{L}_V) = \left\{ 
    f \in L^2(\mathbb{R}^d) : \int_{0}^\infty |\log(\lambda)|^2 \, d\mu_{f,f}^V(\lambda) < \infty \right\}.
\end{equation*}

When $V(x) = |x|^2$, $x\in\mathbb{R}^d$, the operator $\mathcal{L}_V$ reduces to the harmonic oscillator operator $H$. In this case, the spectrum $\sigma(H) = \{2k+d\}_{k\in \mathbb{N}_0}$ is discrete. Here $\mathbb{N}_0 = \mathbb{N} \cup \{0\}$. Moreover, for every $k\in \mathbb{N}_0$ and $\alpha = (\alpha_1,\dots,\alpha_d)\in \mathbb{N}_0^d$ such that $\alpha_1+\dots+\alpha_d = k$, 
\begin{equation*}
    H h_\alpha = (2k+d) h_\alpha,
\end{equation*}
where $\displaystyle h_\alpha(x) = \prod_{j=1}^d h_{\alpha_j}(x_j)$, $x=(x_1,\dots,x_d)\in\mathbb{R}^d$, being
\begin{equation*}
    h_\ell (u) = (-1)^\ell (\sqrt{\pi} 2^\ell \ell !)^{-1/2} \frac{ \partial^\ell}{\partial u^\ell} \left(e^{-u^2}\right) e^{u^2/2}, \quad u\in\mathbb{R} \text{ and } \ell\in\mathbb{N}_0.
\end{equation*}
If $\alpha = (\alpha_1,\dots,\alpha_d)\in \mathbb{N}_0^d$, we denote $s(\alpha) = \alpha_1 +\dots + \alpha_d$. We have that
\begin{equation*}
    (\log H) f = \sum_{k\in\mathbb{N}_0} \log(2k+d) \sum_{\substack{\alpha\in \mathbb{N}_0^d \\ s(\alpha)=k}} c_\alpha(f) h_\alpha, \quad \text{for every } f \in D(\log H).
\end{equation*}
Here, 
\begin{equation*}
    D(\log H) =  \left\{ f \in L^2(\mathbb{R}^d): 
    \sum_{k\in\mathbb{N}_0} \sum_{\substack{\alpha\in \mathbb{N}_0^d \\ s(\alpha)=k}} 
    |\log(2k+d)|^2 |c_\alpha(f)|^2 < \infty
    \right\}
\end{equation*}
and, for $\alpha \in \mathbb{N}_0^d$,
\begin{equation*}
    c_\alpha(f) = \int_{\mathbb{R}^d} h_\alpha(x) f(x) \, dx, \quad f \in L^2(\mathbb{R}^d).
\end{equation*}

For $s \in (0,1)$ we define the $s$-power of $\mathcal{L}_V$ by
\begin{equation*}
    \mathcal{L}_V^s f = \int_0^\infty \lambda^s \, dE_V(\lambda) f, \quad f \in D(\mathcal{L}_V^s),
\end{equation*}
where
\begin{equation*}
    D(\mathcal{L}_V^s) = \left\{
    f \in L^2(\mathbb{R}^d) : \int_0^\infty \lambda^{2s} \, d\mu_{f,f}^V(\lambda) < \infty
    \right\}.    
\end{equation*}
Note that $D(\mathcal{L}_V^{s_2}) \subseteq D(\mathcal{L}_V^{s_1})$ 
provided that $0 < s_1 \leq s_2 < 1$. If $f \in D(\mathcal{L}_V^\beta)$ for some $\beta \in (0,1)$, we have that
\begin{equation*}
    \lim_{s \rightarrow 0^+} \mathcal{L}_V^s f = f \quad \text{in } L^2(\mathbb{R}^d). 
\end{equation*}

   We define the operator $\LLog \mathcal{L}_V$ as follows:
\begin{equation*}
    (\LLog \mathcal{L}_V) f = \frac{d}{ds}(\mathcal{L}_V^s f)\big|_{s=0} = \lim_{s\rightarrow 0^+} \frac{\mathcal{L}_V^s f - f}{s},
\end{equation*}
provided that 
\begin{equation*}
    f \in D(\LLog \mathcal{L}_V) = \left\{ f \in L^2(\mathbb{R}^d) : f \in D(\mathcal{L}_V^\beta) \text{ for some } \beta \in (0,1) \text{ and } \lim_{s\rightarrow 0^+} \frac{\mathcal{L}_V^s f - f}{s} \text{ exists in } L^2(\mathbb{R}^d) \right\}.
\end{equation*}
We define, for every $t > 0$, the operator $T_t^V$ by 
\begin{equation*}
    T_t^V f = \int_0^\infty e^{-\lambda t} \, dE_V(\lambda) f, \quad f \in L^2(\mathbb{R}^d).
\end{equation*}

The family $\{T_t^V\}_{t>0}$ is a $C_0$-semigroup of operators on $L^2(\mathbb{R}^d)$. The operator $\mathcal{L}_V$ is the infinitesimal generator of $\{T_t^V\}_{t>0}$. We can write, for every $t > 0$, $T_t^V = \phi_t(\mathcal{L}_V)$, where $\phi_t(\lambda) = e^{-\lambda t}$, $\lambda \in (0,\infty)$.

The Frullani integral representation (see~\cite[p. 363]{GR}) allows us to write
\begin{equation*}
    \log(\lambda) = \int_0^\infty \frac{e^{-t} - e^{-\lambda t}}{t} \, dt, \quad \lambda > 0. 
\end{equation*}
This formula motivates the following definition:
\begin{equation*}
    (\Log \mathcal{L}_V) f = \lim_{m \rightarrow \infty} \int_{1/m}^m \frac{e^{-t}I - T_t^V}{t} f \, dt, \quad f \in D(\Log \mathcal{L}_V),
\end{equation*}
where the integrals are understood in the $L^2(\mathbb{R}^d)$-Bochner sense and the limit is in $L^2(\mathbb{R}^d)$. Here
$D(\Log \mathcal{L}_V)$ is the set of all functions $f \in L^2(\mathbb{R}^d)$ such that $\displaystyle \frac{e^{-t} I - T_t^V}{t} f$ is $L^2(\mathbb{R}^d)$-Bochner integrable on $(1/m,m)$ for every $m \in \mathbb{N}$ and the limit
\begin{equation*}
    \lim_{m \rightarrow \infty} \int_{1/m}^m \frac{e^{-t} I - T_t^V}{t} f \, dt
\end{equation*}
exists in $L^2(\mathbb{R}^d)$. We recall that $\displaystyle \frac{e^{-t} I - T_t^V}{t} f$ is $L^2(\mathbb{R}^d)$-Bochner integrable on $(1/m,m)$ for $m \in \mathbb{N}$ when it is $L^2(\mathbb{R}^d)$-strongly measurable on $((1/m,m), dt)$ and 
\begin{equation*}
    \int_{1/m}^m \left\| \frac{e^{-t} I - T_t^V}{t} f \right\|_{L^2(\mathbb{R}^d)} \, dt < \infty.
\end{equation*}

The following theorem establishes the equivalence between the different definitions of the logarithmic operator.

\begin{thm}\label{thm: 1.1}
    Let $f \in L^2(\mathbb{R}^d)$. We have that
    \begin{enumerate}
        \item\label{item: teo 1.1 - a} $f \in D(\mathcal{L}_V^{s_0}) \cap D(\log \mathcal{L}_V)$ for some $s_0 \in (0,1)$, if and only if $f \in D(\LLog \mathcal{L}_V)$ and $(\log \mathcal{L}_V) f = (\LLog \mathcal{L}_V) f$.
        \item\label{item: teo 1.1 - b} If $f \in D(\log \mathcal{L}_V)$, then $f \in D(\Log \mathcal{L}_V)$ and $(\log \mathcal{L}_V) f = (\Log \mathcal{L}_V) f$.
    \end{enumerate}
\end{thm}

Our next objective is to obtain a pointwise representation of $\log \mathcal{L}_V$ on $C_c^\infty(\mathbb{R}^d)$. Note that, as will be established in Proposition~\ref{prop: 2.1}, we have $C_c^\infty(\mathbb{R}^d) \subseteq D(\log \mathcal{L}_V)$. 

A key difference from the classical case arises from the semigroup properties: unlike the classical heat semigroup $\{W_t\}_{t>0}$ generated by $-\Delta$, which satisfies $W_t 1 = 1$ for all $t > 0$ (the Markovian property), the semigroup $\{T_t^V\}_{t>0}$ is not Markovian, that is, $T_t^V 1 \neq 1$ for $t > 0$. This fundamental difference produces significant distinctions between the pointwise representation for $\log(-\Delta)$ established in~\cite[Theorem 1.1]{ChW} and the one we shall derive for $\log(\mathcal{L}_V)$.

In order to establish our pointwise representation of $\log \mathcal{L}_V$, we consider a special class of potentials satisfying reverse H\"older inequalities. For $1 < q < \infty$, we say that a locally integrable function $w$ is in the $q$-reverse H\"older class $\RH_q$ when there exists $C > 0$ such that
\begin{equation*}
    \left(\frac{1}{|B|} \int_B w^q \, dx\right)^{1/q} \leq \frac{C}{|B|} \int_B w \, dx,
\end{equation*}
for every ball $B$ in $\mathbb{R}^d$. Functions in $\RH_q$ classes were first studied by Gehring (\cite{G}) and Coifman--Fefferman (\cite{CF}). The condition $V \in \RH_q$ for some $1 < q < \infty$ naturally appears in the study of harmonic analysis operators in the Schr\"odinger setting, such as Riesz transforms, Littlewood--Paley functions, spectral multipliers, commutators, and others (see~\cite{BFHR, BHQ, Dz1, Dz2, DGMTZ, GHY, Shen95, ZT},
and the references therein). 

Suppose that $d \geq 3$ and $V \in \RH_q$ with $q > d/2$. As in~\cite[Eq.~(0.14)]{Shen95}, we define the function $\rho: \mathbb{R}^d \rightarrow (0,\infty)$ by
\begin{equation*}
    \rho(x) = \sup\left\{ r > 0 : \frac{1}{r^{d-2}} \int_{B(x,r)} V(y) \, dy \leq 1 \right\}.
\end{equation*}

The function $\rho$ was introduced in~\cite{Shen94} for potentials $V$ satisfying, for some $C > 0$, 
\begin{equation*}
    \max_{x \in B} V(x) \leq \frac{C}{|B|} \int_B V(x) \, dx,
\end{equation*}
for every ball $B$ in $\mathbb{R}^d$, in the context of studying the Neumann problem for $\mathcal{L}_V$. This auxiliary function $\rho$ plays a fundamental role in harmonic analysis in the Schr\"odinger setting.

It is also proved in~\cite{Shen94} that this function $\rho$ satisfies the following estimates: \begin{equation}\label{eq: rox_vs_roy} c_\rho^{-1} \, \rho(x)\left(1+\frac{|x-y|}{\rho(x)}\right)^{-N_0} \ \leq \ \rho(y) \ \leq \ c_\rho\,\rho(x)\left(1+\frac{|x-y|}{\rho(x)}\right)^{\frac{N_0}{N_0+1}}, \end{equation} for some positive constants $c_\rho$ and $N_0$ independent of $x$ and $y$.
As a consequence of \eqref{eq: rox_vs_roy}, we have in particular that $\rho(x)\sim \rho(y)$ whenever $|x-y|\leq \rho(x)$. That is, there exists a constant $C>0$, independent of $x$ and $y$, such that
$
C^{-1}\,\rho(x) \leq \rho(y) \leq C\,\rho(x)
$ whenever $|x-y|\leq \rho(x)$.

For every $t > 0$ there exists a measurable function $T_t^V : \mathbb{R}^d \times \mathbb{R}^d \rightarrow (0,\infty)$ such that
\begin{equation}\label{eq: 1.1}
    T_t^V(f)(x) = \int_{\mathbb{R}^d} T_t^V(x,y) f(y) \, dy,
\end{equation}
for every $f \in L^2(\mathbb{R}^d)$. Moreover,~\eqref{eq: 1.1} allows us to define a semigroup of contractions in $L^p(\mathbb{R}^d)$, ${1 \leq  p < \infty}$. Since $V$ is nonnegative, the Feynman--Kac formula implies that
\begin{equation}\label{eq: 1.2}
    0 \leq T_t^V(x,y) \leq T_t(x-y), \quad x, y \in \mathbb{R}^d \text{ and } t > 0,
\end{equation}
where $T_t(z) = (4\pi t)^{-d/2} e^{-|z|^2/(4t)}$, $z \in \mathbb{R}^d$ and $t > 0$, is the classical heat kernel.

Estimate~\eqref{eq: 1.2} can be improved using that $V \in \RH_q$ with $q > d/2$. As stated in \cite[Proposition~2]{DGMTZ} (see also \cite{DzZ2, Kur}), for every $N \in \mathbb{N}$ there exists $C_N> 0$ such that 
\begin{equation}\label{eq: 1.3}
    T_t^V(x,y) \leq \frac{C_N}{t^{d/2}} e^{-\frac{|x-y|^2}{5t}} \left( 1 + \frac{\sqrt{t}}{\rho(x)} + \frac{\sqrt{t}}{\rho(y)}\right)^{-N}, 
\end{equation}
for all $x, y \in \mathbb{R}^d$ and $t > 0$. Also,  \cite[Proposition~4.11]{DzZ2} provides a regularity estimate: there exist $\eta, c > 0$ such that for every $N \in \mathbb{N}$ there exists a constant $C_N > 0$ for which
\begin{equation}\label{eq: 1.4}
    |T_t^V(x+h,y) - T_t^V(x,y)| \leq \frac{C_N}{t^{d/2}} \left(\frac{|h|}{\sqrt{t}} \right)^{\eta} e^{-c\frac{|x-y|^2}{t}} \left( 1 + \frac{\sqrt{t}}{\rho(x)} + \frac{\sqrt{t}}{\rho(y)}\right)^{-N}, 
\end{equation}
whenever $|h| < \sqrt{t}$, $x, y \in \mathbb{R}^d$ and $t > 0$.

The operator $\mathcal{L}_V$ can be seen as a nice perturbation of the Euclidean Laplacian $-\Delta$ in the following sense: there exists a nonnegative Schwartz class function $\omega$ on $\mathbb{R}^d$ and $\delta > 0$ such that 
\begin{equation}\label{eq: 1.5}
    |T_t^V(x,y) - T_t(x-y)| \leq \begin{cases}
\left(\dfrac{\sqrt{t}}{\rho(x)}\right)^\delta \omega_t(x-y), & \text{if } \sqrt{t} \leq \rho(x),  \\
\omega_t(x-y), & \text{if } \sqrt{t} > \rho(x),
\end{cases}
\end{equation}
for $x, y \in \mathbb{R}^d$, where $\omega_t(z) = t^{-d/2} \omega(z/\sqrt{t})$, $z \in \mathbb{R}^d$ and $t > 0$ (see \cite[Proposition~2.16]{DzZ3}).
We now establish our pointwise representation for $\log(\mathcal{L}_V)$. This result extends, in particular, \cite[Theorem~1.1]{Feu1} where $V \equiv 1$ is considered. 

\begin{thm}\label{thm: 1.2} 
    Let $d \geq 3$ and $q > d/2$. Suppose that $V \in \RH_q$ and $f \in C^\infty_c(\mathbb{R}^d)$. 
    \begin{enumerate}
        \item \label{item: teo 1.2 - a} We have that
        \begin{equation}\label{eq: teo 1.2}
        \begin{split}
            (\log\mathcal{L}_V)f(x) & = -\int_{B(x,1)} (f(y)-f(x)) \int_0^\infty \frac{T_t^V(x,y)}{t} \, dt \, dy \\
            & \quad - \int_{\mathbb{R}^d \setminus B(x,1)} f(y) \int_0^\infty \frac{T_t^V(x,y)}{t} \, dt \, dy - K(x)f(x),
        \end{split}
        \end{equation}
        for almost every $x \in \mathbb{R}^d$, where
        \begin{equation*}
        \begin{split}
            K(x) & = 2 \log \rho(x) + \int_0^{\rho^2(x)} \int_{\mathbb{R}^d} \frac{T_t^V(x,y) - T_t(x-y)}{t} \, dy \, dt \\
            & \quad - \int_0^{\rho^2(x)} \int_{\mathbb{R}^d \setminus B(x,1)} \frac{T_t^V(x,y)}{t} \, dy \, dt \\
            & \quad + \int_{\rho^2(x)}^\infty \int_{B(x,1)} \frac{T_t^V(x,y)}{t} \, dy \, dt +\gamma,
        \end{split}
        \end{equation*}
        and $\gamma$ denotes the Euler-Mascheroni constant.

        \item \label{item: teo 1.2 - b} For every $1 < p \leq \infty$,
        \begin{equation*}
            \lim_{s \rightarrow 0^{+}} \frac{\mathcal{L}_V^s f - f}{s} = (\log \mathcal{L}_V) f, 
        \end{equation*}
        in $L^p(\mathbb{R}^d)$.
    \end{enumerate}
\end{thm}

Note that the function $K$ is not constant. This reflects precisely the difference mentioned earlier: the semigroup $\{T_t^V\}_{t>0}$ is not Markovian, unlike the classical heat semigroup. Consequently, the constant term in the pointwise representation of $\log \mathcal{L}_V$ becomes a non-trivial function $K(x)$ depending on the auxiliary function $\rho(x)$, which encodes information about the potential $V$. This fact, combined with the lack of an explicit formula for the heat kernel $T_t^V(x,y)$---we only have the estimates \eqref{eq: 1.2}, \eqref{eq: 1.3}, \eqref{eq: 1.4} and \eqref{eq: 1.5}---makes the proof of Theorem~\ref{thm: 1.2} substantially more intricate than that of \cite[Theorem~1.1]{ChW} for the Euclidean case.
A small modification of the proof of Theorem~\ref{thm: 1.2} allows us to see that if $f \in \Lip^\theta(\mathbb{R}^d)$, the $\theta$-Lipschitz space on $\mathbb{R}^d$, for some $\theta \in (0,1]$, and 
\begin{equation*}
    \int_{\mathbb{R}^d} \frac{|f(y)|}{(1+|y|)^{d}} \, dy < \infty,
\end{equation*}
then
\begin{equation*}
\begin{split}
    \lim_{s \rightarrow 0^+} \frac{1}{s} \left( \frac{1}{\Gamma(-s)} \int_0^\infty \frac{T_t^V f(x) - f(x)}{t^{s+1}} \, dt - f(x) \right) & = - \int_{B(x,1)} (f(y)-f(x)) \int_0^\infty \frac{T_t^V(x,y)}{t} \, dt \, dy \\
    & \quad -\int_{\mathbb{R}^d \setminus B(x,1)} f(y) \int_0^\infty \frac{T_t^V(x,y)}{t} \, dt \, dy - K(x)f(x),
\end{split}
\end{equation*}
for $x \in \mathbb{R}^d$. Thus, the domain of $\log \mathcal{L}_V$ can be extended according to the last identity.

Suppose now $V(x) = 1$, $x \in \mathbb{R}^d$. In this case $T_t^V = e^{-t} T_t$, $t > 0$. By manipulating formula~\eqref{eq: teo 1.2} we get for $f$ as in Theorem~\ref{thm: 1.2}:
\begin{equation*}
    \begin{split}
        (\log \mathcal{L}_V) f(x) & = -\int_{\mathbb{R}^d} (f(y)-f(x)) \int_0^\infty \frac{e^{-t}T_t(x-y)}{t} \, dt \, dy \\
        & \quad - f(x) \left( 2\log\rho(x) + \int_0^{\rho^2(x)} \frac{e^{-t} - 1}{t} \, dt + \int_{\rho^2(x)}^\infty \frac{e^{-t}}{t} \, dt - \Gamma'(1) \right).
    \end{split}
\end{equation*}
Since
\begin{equation*}
    \log(z) + \int_0^z \frac{e^{-t} - 1}{t} \, dt + \int_z^\infty \frac{e^{-t}}{t} \, dt = -\gamma \quad \text{for } z \in (0,\infty),
\end{equation*}
and $\Gamma'(1) = -\gamma$, we conclude that 
\begin{equation*}
    (\log \mathcal{L}_V) f(x) = -\int_{\mathbb{R}^d} (f(y)-f(x)) \int_0^\infty \frac{e^{-t}}{t} T_t(x-y) \, dt \, dy, \quad \text{a.e. } x \in \mathbb{R}^d.
\end{equation*}

On the other hand, we have that
\begin{equation*}
    \int_0^\infty \frac{e^{-t}}{t} T_t(x-y) \, dt = \int_0^\infty \frac{e^{-t - \frac{|x-y|^2}{4t}}}{(4\pi t)^{d/2} t} \, dt = \frac{2}{(4\pi)^{d/2}} \left(\frac{|x-y|}{2}\right)^{-d/2} \mathcal{K}_{d/2}(|x-y|), \quad x, y \in \mathbb{R}^d,
\end{equation*}
where $\mathcal{K}_\nu$ denotes the modified Bessel function of the second kind of order $\nu$. Thus, \cite[Theorem~1.1]{Feu1} is a special case of Theorem~\ref{thm: 1.2} here.

Fall and Felli (\cite{FaFe2} and \cite{FaFe1}) studied the essential self-adjointness and established unique continuation properties for a relativistic Schr\"odinger operator with a singular homogeneous potential defined by 
\begin{equation*}
    H = (-\Delta + m^2)^s - \frac{a(x/|x|)}{|x|^{2s}} - h(x), 
\end{equation*}
where $m \geq 0$ and $s \in (0,1)$. In~\cite[Eq.~(1.3)]{FaFe1} the authors give a pointwise representation for the operator $(-\Delta + m^2)^s$. A pointwise representation for the operator $\log(-\Delta + m^2)$ appears here as a special case of Theorem~\ref{thm: 1.2}.

In~\cite{BHS} and~\cite{MSTZ}, Lipschitz spaces in the Schr\"odinger setting were introduced as follows: we say that a function $f \in \Lip_V^\alpha$ with $0 < \alpha < 1$ when 
\begin{equation*}
    \|\rho^{-\alpha} f\|_\infty < \infty \quad \text{and} \quad \sup_{|h|>0} \frac{\|f(\cdot + h) - f(\cdot)\|_\infty}{|h|} < \infty .
\end{equation*}
The definition of $\Lip_V^\alpha$ was extended to $0 < \alpha < 2$ in~\cite{dLT}.

For every $\alpha > 0$ we consider the negative fractional power $\mathcal{L}_V^{-\alpha}$ of $\mathcal{L}_V$ defined by
\begin{equation*}
    \mathcal{L}_V^{-\alpha} f = \int_0^\infty \lambda^{-\alpha} \, dE_V(\lambda) f, \quad f \in D(\mathcal{L}_V^{-\alpha}), 
\end{equation*}
where
\begin{equation*}
    D(\mathcal{L}_V^{-\alpha}) = \left\{ f \in L^2(\mathbb{R}^d) : \int_0^\infty \lambda^{-2\alpha} \, d\mu_{f,f}^V(\lambda) < \infty \right\}.
\end{equation*}
We are going to see in Proposition~\ref{propo: 2.5} that $\mathcal{L}_V^{-\alpha}$ can be expressed in terms of the semigroup $\{T_t^V\}_{t>0}$ as 
\begin{equation}\label{eq: 1.6}
    \mathcal{L}_V^{-\alpha} f = \frac{1}{\Gamma(\alpha)} \int_0^\infty T_t^V(f) t^{\alpha-1} \, dt,
\end{equation}
By using~\eqref{eq: 1.1} we can see that the integral in~\eqref{eq: 1.6} allows us to extend $\mathcal{L}_V^{-\alpha}$ to $L^p(\mathbb{R}^d)$ as a bounded operator from $L^p(\mathbb{R}^d)$ into $L^q(\mathbb{R}^d)$ where $1/q = 1/p - 2\alpha/d$, $1 < p < d/(2\alpha)$, and from $L^1(\mathbb{R}^d)$ into $L^{d/(d-2\alpha), \infty}(\mathbb{R}^d)$. As can be seen in equation~\eqref{eq: 1.2}, the integral kernel of the operator $T_t^V$ has a better behavior far away from the diagonal than the classical heat kernel $T_t$. In~\cite{DGMTZ} it was proved that $\mathcal{L}_V^{-\alpha}$ maps $L^{d/(2\alpha)}(\mathbb{R}^d)$ into a $\textup{BMO}$-space denoted by $\textup{BMO}_{\mathcal{L}_V}$ adapted to the Schr\"odinger setting. In \cite[Theorem 1.2]{MSTZ} (see also~\cite[Theorem 1.6]{dLT}) the authors proved that $\mathcal{L}_V^{-\alpha}$ maps $\Lip_V^\theta$ into $\Lip_V^{\theta - 2\alpha}$ when $0 < \theta < 1$, $\alpha > 0$ and $\theta + 2\alpha < 1$.

Suppose now that $0$ does not belong to the spectrum $\sigma(\mathcal{L}_V)$. Then, there exists $a > 0$ such that $\sigma(\mathcal{L}_V) \subseteq [a, +\infty)$. Since the function $\phi(\lambda) = \lambda^{-\alpha}$, $\lambda \in (0,+\infty)$ is bounded on $[a,+\infty)$, $\mathcal{L}_V^{-\alpha}$ is a bounded operator on $L^2(\mathbb{R}^d)$, for every $\alpha > 0$. Moreover, $\mathcal{L}_V^{-\alpha} \circ \mathcal{L}_V^{-\beta} = \mathcal{L}_V^{-(\alpha+\beta)}$, for every $\alpha, \beta > 0$, and $\lim_{\alpha \rightarrow 0^+} \mathcal{L}_V^{-\alpha} f = f$ in $L^2(\mathbb{R}^d)$ for every $f \in L^2(\mathbb{R}^d)$.

The family $\{\mathcal{L}_V^{-\alpha}\}_{\alpha>0}$ is a $C_0$-semigroup of operators on $L^2(\mathbb{R}^d)$. We will prove in Proposition~\ref{propo: 2.3} that the infinitesimal generator of $\{\mathcal{L}_V^{-\alpha}\}_{\alpha>0}$ is the logarithmic operator $\log \mathcal{L}_V$. We recall here that $\sigma(H) = \{2k+d\}_{k\in \mathbb{N}_0}$ when $H$ represents the harmonic oscillator operator.

Our final objective is to study the following Cauchy problem

\begin{equation}\label{eq: 1.7}
    \begin{cases}
        \dfrac{\partial u}{\partial t} (x,t) + (\log \mathcal{L}_V) (u(\cdot,t))(x) = 0, & x \in \mathbb{R}^d, \; t > 0,\\
        u(x,0) = f(x), & x \in \mathbb{R}^d,
    \end{cases}
\end{equation}
in $\mathbb{R}^d \times (0,\infty)$. If $0 \notin \sigma(\mathcal{L}_V)$, then, for every $f \in L^2(\mathbb{R}^d)$, $u(x,t) = \mathcal{L}_V^{-t} (f)(x)$ is an $L^2$-solution of~\eqref{eq: 1.7}.

\begin{thm}\label{thm: 1.3}
    Let $d \geq 3$ and $V \in \RH_q$ with $q > d/2$. Suppose that $f \in L^1(\mathbb{R}^d) \cap \Lip_V^\theta \cap C^1(\mathbb{R}^d)$, for some $\theta \in (0,1)$, and $\nabla f \in L^1(\mathbb{R}^d)$. Then, the function $u$ defined by
    \begin{equation*}
        u(x,t) = \frac{1}{\Gamma(t)} \int_0^\infty T_z^V (f)(x) z^{t-1} \, dz, \quad x \in \mathbb{R}^d \text{ and } 0 < t < 1 - \theta,
    \end{equation*}
    satisfies the following properties:
    \begin{enumerate}
        \item $u$ and $\dfrac{\partial u}{\partial t}$ are continuous functions in $\mathbb{R}^d \times (0,1-\theta)$. 
        \item $\dfrac{\partial u}{\partial t}(x,t) = -(\log \mathcal{L}_V)(u(\cdot,t))(x)$, for $(x,t) \in \mathbb{R}^d \times (0,1-\theta)$.  
        \item For every $x \in \mathbb{R}^d$, $\displaystyle \lim_{t\rightarrow 0^+} u(x,t) = f(x)$.
    \end{enumerate}
\end{thm}

We prove Theorem~\ref{thm: 1.1} in Section~\ref{sec: 2}, Theorem~\ref{thm: 1.2} in Section~\ref{sec: 3} and Theorem~\ref{thm: 1.3} in Section~\ref{sec: 4}. Throughout this paper $C$ and $c$ always represent positive constants that may change in each occurrence.

\section{About the logarithmic operator of \texorpdfstring{$\mathcal{L}_V$}{LV}}\label{sec: 2}

In this section we prove Theorem~\ref{thm: 1.1} together with some properties of the logarithmic operator of $\mathcal{L}_V$. As we mentioned before, we define 
\[
(\log \mathcal{L}_V) f = \int_0^\infty \log(\lambda) \, dE_V(\lambda) f, \quad f \in D(\log \mathcal{L}_V),
\]
where
\begin{equation*}
    D(\log \mathcal{L}_V) = \left\{ 
    f \in L^2(\mathbb{R}^d) : \int_{0}^\infty |\log(\lambda)|^2 \, d\mu_{f,f}^V(\lambda) < \infty \right\}.
\end{equation*}

Here $E_V$ denotes the spectral measure of $\mathcal{L}_V$ and, for every $f \in L^2(\mathbb{R}^d)$, $\mu_{f,f}^V$ is the Borel measure defined by $\mu_{f,f}^V(A) = \langle E_V(A) f, f \rangle$ for each Borel set $A \subseteq (0,\infty)$.

\begin{prop}\label{prop: 2.1}
    Let $d \geq 3$ and $V \in \RH_q$ with $q > d/2$. Then $D(\mathcal{L}_V^{-\beta}) \cap D(\mathcal{L}_V^{\alpha}) \subseteq D(\log \mathcal{L}_V)$, for every $\alpha, \beta \in (0,1)$, and $C_c^\infty (\mathbb{R}^d) \subseteq D(\log \mathcal{L}_V)$.
\end{prop}
\begin{proof}
    Let $f \in L^2(\mathbb{R}^d)$ and $\alpha, \beta \in (0,1)$. We can write
    \begin{equation*}
        \begin{split}
            \int_0^\infty (\log \lambda)^2 \, d\mu_{f,f}^V(\lambda) & \leq C 
            \left( \int_0^1 \lambda^{-2\beta} \, d\mu_{f,f}^V(\lambda) + \int_1^\infty \lambda^{2\alpha} \, d\mu_{f,f}^V(\lambda)\right) \\
            & \leq C 
            \left( \int_0^\infty \lambda^{-2\beta} \, d\mu_{f,f}^V(\lambda) + \int_0^\infty \lambda^{2\alpha} \, d\mu_{f,f}^V(\lambda)\right).
        \end{split}
    \end{equation*}
    We deduce that $f \in D(\log \mathcal{L}_V)$ provided that $f \in D(\mathcal{L}_V^{-\beta}) \cap D(\mathcal{L}_V^{\alpha})$.

    On the other hand, $D(\mathcal{L}_V) \subseteq D(\mathcal{L}_V^\alpha)$. Hence, $C_c^\infty(\mathbb{R}^d) \subseteq D(\mathcal{L}_V^\alpha)$. Moreover, $\mathcal{L}_V^{-\beta}$ is bounded from $L^p(\mathbb{R}^d)$ into $L^q(\mathbb{R}^d)$ when $1/q = 1/p - (2\beta)/d$ and $1 < p < d/(2\beta)$. In particular, $\mathcal{L}_V^{-\beta}$ is bounded from $L^{(2d)/(d+2\beta)}(\mathbb{R}^d)$ into $L^2(\mathbb{R}^d)$, so $L^{(2d)/(d+2\beta)}(\mathbb{R}^d) \subseteq D(\mathcal{L}_V^{-\beta})$. Since $C_c^\infty(\mathbb{R}^d) \subseteq L^{(2d)/(d+2\beta)}(\mathbb{R}^d)$, we conclude that $C_c^\infty(\mathbb{R}^d) \subseteq D(\mathcal{L}_V^{-\beta}) \cap D(\mathcal{L}_V^\alpha) \subseteq D(\log \mathcal{L}_V)$.
\end{proof}

We define, for every $\beta \in \mathbb{R}$, $J_\beta = \mathcal{L}_V^{i\beta}=\phi_\beta(\mathcal{L}_V)$, where $\phi_\beta(\lambda)=\lambda^{i\beta}$, $\lambda\in (0,\infty)$. Since this function is bounded for $\lambda \in (0,\infty)$, the operator $J_\beta$ is bounded on $L^2(\mathbb{R}^d)$ for every $\beta \in \mathbb{R}$. We have that 
\[
J_\alpha \circ J_\beta = J_{\alpha + \beta}, \quad \alpha, \beta \in \mathbb{R}.
\]
Hence $\{J_\beta\}_{\beta \in \mathbb{R}}$ is a group of bounded operators on $L^2(\mathbb{R}^d)$.
\begin{prop}\label{propo: 2.2}
    The operator $i\log \mathcal{L}_V$ is the infinitesimal generator of the group $\{J_\beta\}_{\beta \in \mathbb{R}}$ in $L^2(\mathbb{R}^d)$. 
\end{prop}
\begin{proof}
    We denote by $G$ the infinitesimal generator of $\{J_\beta\}_{\beta \in \mathbb{R}}$ in $L^2(\mathbb{R}^d)$, that is, 
    \[
    Gf = \lim_{h \rightarrow 0} \frac{J_h f - f}{h}, \quad \text{in } L^2(\mathbb{R}^d),
    \]
    for every $\displaystyle f \in D(G) = \left\{ g \in L^2(\mathbb{R}^d) : \lim_{h \rightarrow 0} \frac{J_h g - g}{h} \text{ exists in } L^2(\mathbb{R}^d) \right\}$.

    Let $f \in D(\log \mathcal{L}_V)$. We are going to show that
    \[
    \lim_{h \rightarrow 0} \frac{J_h f - f}{h} = i(\log \mathcal{L}_V) f, \quad \text{in } L^2(\mathbb{R}^d),
    \]
    or equivalently that
    \[
    \left\| \frac{J_h f - f}{h} - i(\log \mathcal{L}_V) f \right\|_{L^2(\mathbb{R}^d)}^2 = 
    \int_0^\infty \left| \frac{\lambda^{ih}-1}{h} - i\log(\lambda)\right|^2 \, d\mu_{f,f}^V(\lambda) \rightarrow 0 
    \]
    as $h \rightarrow 0$. 

    We note that $\left| \frac{\lambda^{ih}-1}{h}\right| \leq |\log \lambda|$ for all $h \in \mathbb{R}$ and $\lambda > 0$. Then, 
    \[
    \left| \frac{\lambda^{ih}-1}{h} - i\log\lambda \right| \leq 2 |\log\lambda|. 
    \]

    Since $f \in D(\log \mathcal{L}_V)$, the Dominated Convergence Theorem leads to 
    \[
    \lim_{h \rightarrow 0} \int_0^\infty \left| \frac{\lambda^{ih}-1}{h} - i\log\lambda \right|^2 \, d\mu_{f,f}^V(\lambda) = 0. 
    \]
    We conclude that $f \in D(G)$ and $Gf = i(\log\mathcal{L}_V) f$.

    Conversely, suppose that $f \in D(G)$. Then there exists $g \in L^2(\mathbb{R}^d)$ such that $Gf = g$, that is,
    \[
    \lim_{h \rightarrow 0} \left\| \int_0^\infty \frac{\lambda^{hi}-1}{h} \, dE_V(\lambda) f - g \right\|_{L^2(\mathbb{R}^d)} = 0.
    \]

    Let $(h_k)_{k \in \mathbb{N}} \subseteq (0,\infty)$ be a sequence converging to zero. The sequence
    \[
    \left\{ \int_0^\infty \frac{\lambda^{h_k i} - 1}{h_k} \, dE_V(\lambda) f \right\}_{k \in \mathbb{N}}
    \]
    is Cauchy in $L^2(\mathbb{R}^d)$, or equivalently, the sequence $\displaystyle \left\{ \frac{\lambda^{h_k i}-1}{h_k}\right\}_{k \in \mathbb{N}}$ is Cauchy in $L^2((0,\infty), d\mu_{f,f}^V)$. Hence, there exists $h \in L^2((0,\infty), d\mu_{f,f}^V)$ such that
    \[
    \lim_{k \rightarrow \infty} \frac{\lambda^{h_k i}-1}{h_k} = h(\lambda) \quad \text{in } L^2((0,\infty), d\mu_{f,f}^V).
    \]
    On the other hand, for each $\lambda > 0$,
    \[
    \lim_{k \rightarrow \infty} \frac{\lambda^{h_k i}-1}{h_k} = i\log\lambda.
    \]
    Therefore, $h(\lambda) = i\log\lambda$ for $\mu_{f,f}^V$-almost every $\lambda > 0$.

    We have proved that 
    \[ 
    \left\| \int_0^\infty \frac{\lambda^{h_k i}-1}{h_k} \, dE_V(\lambda) f - \int_0^\infty i\log\lambda \, dE_V(\lambda) f \right\|_{L^2(\mathbb{R}^d)}^2 = 
    \int_0^\infty \left| \frac{\lambda^{h_k i}-1}{h_k} - i\log\lambda \right|^2 \, d\mu_{f,f}^V(\lambda) \rightarrow 0
    \]
    as $k \rightarrow \infty$. The arbitrariness of the sequence $(h_k)_{k \in \mathbb{N}}$ allows us to conclude that $f \in D(\log \mathcal{L}_V)$ and $g = i(\log \mathcal{L}_V) f$.
\end{proof}

By proceeding in a similar way we can prove the following property:

\begin{prop}\label{propo: 2.3}
    Suppose that $0 \notin \sigma(\mathcal{L}_V)$. Then $-\log \mathcal{L}_V$ is the infinitesimal generator of the $C_0$-semigroup $\{\mathcal{L}_V^{-t}\}_{t>0}$ of operators on $L^2(\mathbb{R}^d)$.
\end{prop}

We proceed now with the proof of Theorem~\ref{thm: 1.1}.

\begin{proof}[Proof of Theorem~\ref{thm: 1.1}]
We begin proving part~\ref{item: teo 1.1 - a}. Suppose that $f \in D(\mathcal{L}_V^{s_0}) \cap D(\log \mathcal{L}_V)$ for some $s_0 \in (0,1)$. Then $f \in D(\mathcal{L}_V^s)$ for every $0 < s < s_0$. We have that
\[
\frac{|\lambda^s-1|}{s} \leq C \begin{cases}
    |\log \lambda|, & \text{if } 0 < \lambda < 1, \\
    |\log \lambda| \lambda^s \leq C \lambda^{s_0}, & \text{if } 0 < s < s_0 \text{ and } \lambda \geq 1.
\end{cases}
\]

Then, 
\[ 
\left| \frac{\lambda^{s}-1}{s} - \log \lambda\right|^2 \leq C (\lambda^{2s_0} + |\log \lambda|^2)
\]
for $\lambda > 0$ and $0 < s < s_0$. Since $\int_0^\infty \lambda^{2s_0} \, d\mu_{f,f}^V(\lambda) < \infty$ and $\int_0^\infty |\log\lambda|^2 \, d\mu_{f,f}^V(\lambda) < \infty$, the Dominated Convergence Theorem allows us to obtain that
\[
\lim_{s \rightarrow 0^+} 
\left\|
\frac{\mathcal{L}_V^s f - f}{s} - (\log \mathcal{L}_V) f
\right\|_{L^2(\mathbb{R}^d)}^2 = 
\lim_{s \rightarrow 0^+}
\int_0^\infty \left| \frac{\lambda^{s}-1}{s} - \log \lambda\right|^2 \, d\mu_{f,f}^V(\lambda) = 0.
\]

Thus, we have proved that $f \in D(\LLog \mathcal{L}_V)$ and $(\LLog \mathcal{L}_V)f = (\log \mathcal{L}_V) f$.

If $f\in D(\LLog \mathcal{L}_V)$, by proceeding as in the proof of Proposition~\ref{propo: 2.2} we can see that $f\in D(\mathcal{L}_V^{s_0})\cap D(\log \mathcal{L}_V)$, for some $s_0\in (0,1)$, and $(\log \mathcal{L}_V) f = (\LLog \mathcal{L}_V) f$.

To prove part~\ref{item: teo 1.1 - b} we consider the function $F$ defined as follows:
\begin{equation*}
    \begin{aligned}
F: (0,\infty) & \rightarrow L^2(\mathbb{R}^d) \\
t & \mapsto \frac{e^{-t}I - T_t^V}{t} f, \text{ where } f\in L^2(\mathbb R^d). 
\end{aligned}
\end{equation*}

Since $L^2(\mathbb{R}^d)$ is a separable space, according to~\cite[p.~131]{Yo}, to show that $F$ is $L^2(\mathbb{R}^d)$-strongly measurable on $((0,\infty), dt)$ it suffices to prove that $F$ is $L^2(\mathbb{R}^d)$-weakly measurable.

Let $g\in L^2(\mathbb{R}^d)$. We define, for $t>0$,
\begin{equation*}
    H(t) = \int_{\RR^d} [F(t)](x) g(x) dx = \frac{1}{t} \left( e^{-t} \int_{\RR^d} f(x)g(x)dx - \int_{\RR^d} T_t^V(f)(x)g(x)dx\right).
\end{equation*}
Since $\{T_t^V\}_{t>0}$ is a $C_0$-semigroup of operators in $L^2(\mathbb{R}^d)$, the function $H$ is continuous in $(0,\infty)$. Hence, $H$ is a measurable function on $((0,\infty), dt)$. Thus, we proved that $F$ is weakly measurable in $((0,\infty),dt)$.

Let $m \in \mathbb{N}$. We have that 
\[
\int_{1/m}^m \|F(t)\|_{L^2(\mathbb{R}^d)} \, dt < \infty .
\]
Indeed, we can write
\begin{equation*}
    \begin{split}
        \int_{1/m}^m \|F(t)\|_{L^2(\mathbb{R}^d)} \, dt & = \int_{1/m}^m \left( \int_0^\infty \left| \frac{e^{-t} - e^{-t\lambda}}{t} \right|^2 \, d\mu_{f,f}^V(\lambda) \right)^{1/2} dt \\
        & \leq 2 \int_{1/m}^m \frac{1}{t} \left( \int_0^\infty d\mu_{f,f}^V(\lambda) \right)^{1/2} dt \\
        & \leq 4 \log(m) \|f\|_{L^2(\mathbb{R}^d)}.
    \end{split}
\end{equation*}
Hence, $F$ is $L^2(\mathbb{R}^d)$-Bochner integrable on $(1/m, m)$ for each $m\in\NN$.

Suppose now that $f \in D(\log \mathcal{L}_V)$. We are going to show that 
\begin{equation*}
    \lim_{m \rightarrow \infty} \int_{1/m}^m \frac{e^{-t}I - T_t^V}{t} f \, dt = (\log \mathcal{L}_V) f,
\end{equation*}
where the integrals are understood in the $L^2(\mathbb{R}^d)$-Bochner sense and the limit is taken in $L^2(\mathbb{R}^d)$.

Let $m \in \mathbb{N}$ and $g \in L^2(\mathbb{R}^d)$. It follows that
\begin{equation*}
    \begin{split}
        \int_{\mathbb{R}^d} |g(x)| & \int_{1/m}^m \left|
        \left( \int_0^\infty \frac{e^{-t} - e^{-t\lambda}}{t} \, dE_V(\lambda)\right) f(x) \right| dt \, dx \\
        & \leq \|g\|_{L^2(\mathbb{R}^d)} \int_{1/m}^m \left\| \left( \int_0^\infty \frac{e^{-t} - e^{-t\lambda}}{t} \, dE_V(\lambda)\right) f \right\|_{L^2(\mathbb{R}^d)} dt \\
        & \leq \|g\|_{L^2(\mathbb{R}^d)} \int_{1/m}^m 
        \left( \int_0^\infty \left|\frac{e^{-t} - e^{-t\lambda}}{t} \right|^2 \, d\mu_{f,f}^V(\lambda) \right)^{1/2} dt \\
        & \leq 4  \log (m)\|g\|_{L^2(\mathbb{R}^d)} \|f\|_{L^2(\mathbb{R}^d)} \, .
    \end{split}
\end{equation*}

Then, Fubini's theorem leads to 
\begin{equation*}
    \begin{split}
        \int_{\mathbb{R}^d} g(x) \int_{1/m}^m & \left( \int_0^\infty \frac{e^{-t} - e^{-t\lambda}}{t} \, dE_V(\lambda)\right) f(x) \, dt \, dx \\
        & = \int_{1/m}^m \int_{\mathbb{R}^d} g(x) \left( \int_0^\infty \frac{e^{-t} - e^{-t\lambda}}{t} \, dE_V(\lambda)\right) f(x) \, dx \, dt \\
        & = \int_{1/m}^m \int_0^\infty \frac{e^{-t} - e^{-t\lambda}}{t} \, d\mu_{f,g}^V(\lambda) \, dt \\
        & = \int_0^\infty \int_{1/m}^m \frac{e^{-t} - e^{-t\lambda}}{t} \, dt \, d\mu_{f,g}^V(\lambda) \\
        & = \int_{\mathbb{R}^d} g(x) \left( \int_0^\infty \int_{1/m}^m \frac{e^{-t} - e^{-t\lambda}}{t} \, dt \, dE_V(\lambda) \right) f(x) \, dx.
    \end{split}
\end{equation*}

The arbitrariness of $g$ leads to 
\begin{equation*}
    \int_{1/m}^m \left( \int_0^\infty \frac{e^{-t} - e^{-t\lambda}}{t} \, dE_V(\lambda) \right) f \, dt = \int_0^\infty \int_{1/m}^m \frac{e^{-t} - e^{-t\lambda}}{t} \, dt \, dE_V(\lambda) f.
\end{equation*}

Our objective is to show that 
\begin{equation*}
    \lim_{m \rightarrow \infty} \int_0^\infty \int_{1/m}^m \frac{e^{-t} - e^{-t\lambda}}{t} \, dt \, dE_V(\lambda) f = (\log \mathcal{L}_V) f,
\end{equation*}
where the limit is taken in $L^2(\mathbb{R}^d)$.

We have that 
\begin{equation*}
    \begin{split}
        \left\| \int_0^\infty \int_m^\infty \frac{e^{-t} - e^{-t\lambda}}{t} \, dt \, dE_V(\lambda) f \right\|_{L^2(\mathbb{R}^d)}^2 
        & = \int_0^\infty \left( \int_m^\infty \left| \frac{e^{-t} - e^{-t\lambda}}{t} \right| \, dt \right)^2 d\mu_{f,f}^V(\lambda) \\
        & \leq \int_0^1 \left( \int_m^\infty \left| \frac{e^{-t} - e^{-t\lambda}}{t} \right| \, dt \right)^2 d\mu_{f,f}^V(\lambda) \\
        & \quad + \int_1^\infty \left( \int_m^\infty \left| \frac{e^{-t} - e^{-t\lambda}}{t} \right| \, dt \right)^2 d\mu_{f,f}^V(\lambda) \\
        & =: J_1(m) + J_2(m), \quad m \in \mathbb{N}.
    \end{split}
\end{equation*}

We get, as in the proof of \cite[Theorem~2.12]{Ch-manifolds},
\begin{equation*}
    \int_m^\infty \left| \frac{e^{-t} - e^{-t\lambda}}{t} \right| dt = \int_m^\infty \frac{e^{-t\lambda} - e^{-t}}{t} \, dt \leq \int_1^\infty \frac{e^{-t\lambda} - e^{-t}}{t} \, dt \leq C(|\log\lambda| + 1),
\end{equation*}
for $0 < \lambda < 1$ and $m \in \mathbb{N}$.

Since $f \in D(\log \mathcal{L}_V)$, we have $\displaystyle \int_0^1 (|\log\lambda| + 1)^2 \, d\mu_{f,f}^V(\lambda) < \infty$, and the Dominated Convergence Theorem leads to
\[
\lim_{m \rightarrow \infty} J_1(m) = 0.
\]

On the other hand, we have that
\[
\int_m^\infty \left| \frac{e^{-t} - e^{-t\lambda}}{t} \right| \, dt \leq 2 \int_m^\infty \frac{e^{-t}}{t} \, dt \leq 2 \int_1^\infty \frac{e^{-t}}{t} \, dt, \quad \lambda > 1 \text{ and } m \in \mathbb{N}.
\]
Then, since $\mu_{f,f}^V$ is a bounded measure and $\int_1^\infty \frac{e^{-t}}{t} \, dt < \infty$, we obtain
\[
\lim_{m \rightarrow \infty} J_2(m) = 0.
\]

We conclude that 
\[ 
\lim_{m \rightarrow \infty} \left( J_1(m) + J_2(m) \right) = 0.
\]

Again, as in the proof of \cite[Theorem~2.12]{Ch-manifolds} we obtain
\[
\int_0^{1/m} \left| \frac{e^{-t} - e^{-t\lambda}}{t} \right| \, dt \leq \int_0^1 \left| \frac{e^{-t} - e^{-t\lambda}}{t} \right| \, dt \leq C(|\log \lambda| + 1), \quad \lambda > 0 \text{ and } m \in \mathbb{N}. 
\]
The Dominated Convergence Theorem leads to 
\[ 
\lim_{m \rightarrow \infty} \left\| \int_0^\infty \int_0^{1/m} \frac{e^{-t} - e^{-t\lambda}}{t} \, dt \, dE_V(\lambda) f \right\|_{L^2(\mathbb{R}^d)} = 0.
\]

We conclude that $f \in D(\Log \mathcal{L}_V)$ and 
\[ 
\lim_{m \rightarrow \infty} \int_0^\infty \int_{1/m}^m \frac{e^{-t} - e^{-t\lambda}}{t} \, dt \, dE_V(\lambda) f = \int_0^\infty \log(\lambda) \, dE_V(\lambda) f = (\log \mathcal{L}_V) f,
\]
in $L^2(\mathbb{R}^d)$, which completes the proof.

\end{proof}

We define, for every $\alpha \in (0,1)$, the operator 
\[ 
\mathcal{L}_{V,\textup{heat}}^\alpha f = \frac{1}{\Gamma(-\alpha)} \lim_{m \rightarrow \infty} \int_{1/m}^m \frac{T_t^V f - f}{t^{1+\alpha}} \, dt, \quad f \in D(\mathcal{L}_{V,\textup{heat}}^{\alpha}), 
\]
where the integrals are understood in the $L^2(\mathbb{R}^d)$-Bochner sense and the limit is taken in $L^2(\mathbb{R}^d)$. Here
\begin{equation*}
\begin{split}
    D(\mathcal{L}_{V,\textup{heat}}^{\alpha}) = & \left\{ f \in L^2(\mathbb{R}^d) : \frac{T_t^V f - f}{t^{1+\alpha}} \text{ is } L^2(\mathbb{R}^d)\text{-Bochner integrable on } ((1/m,m), dt),\right. \\
    & \left. \text{ for every } m \in \mathbb{N}, \text{ and } \lim_{m \rightarrow \infty} \int_{1/m}^m \frac{T_t^V f - f}{t^{1+\alpha}} \, dt \text{ exists in } L^2(\mathbb{R}^d) \right\}.
\end{split}
\end{equation*}

\begin{prop} \label{propo: 2.4}
    Let $f \in L^2(\mathbb{R}^d)$ and $\alpha \in (0,1)$.
    \begin{enumerate}
        \item\label{item: propo 2.4 - a} If $f \in D(\mathcal{L}_V^\alpha)$, then $f \in D(\mathcal{L}_{V,\textup{heat}}^{\alpha})$ and $\mathcal{L}_{V,\textup{heat}}^{\alpha} f = \mathcal{L}_V^\alpha f$.
        \item\label{item: propo 2.4 - b} If $f \in D(\mathcal{L}_V^\beta)$ for some $\alpha < \beta \leq 1$, then
        \[
        \mathcal{L}_V^\alpha f = \frac{1}{\Gamma(-\alpha)} \int_0^\infty \frac{T_t^V f - f}{t^{1+\alpha}} \, dt,
        \]
        where the integral converges in the $L^2(\mathbb{R}^d)$-Bochner sense.
    \end{enumerate}
\end{prop}

\begin{proof}
    We note first that part~\ref{item: propo 2.4 - a} can be proved proceeding as in the proof of part~\ref{item: teo 1.1 - b} of Theorem~\ref{thm: 1.1}. To prove part~\ref{item: propo 2.4 - b} of Proposition~\ref{propo: 2.4}, we consider the function $F_\alpha$ defined as follows:
\begin{equation*}
    \begin{aligned}
F_\alpha: (0,\infty) & \rightarrow L^2(\mathbb{R}^d) \\
t & \mapsto \frac{e^{-t}I - T_t^V}{t^{1+\alpha}} f, \text{ where } f\in L^2(\mathbb R^d). 
\end{aligned}
\end{equation*}

    Let $f \in D(\mathcal{L}_V^\beta)$ for some $\alpha < \beta \leq 1$. We have that
    \begin{equation*}
        \begin{split}
            \int_0^\infty \|F_\alpha(t)\|_{L^2(\mathbb{R}^d)} \, dt & = \int_0^\infty \left( \int_0^\infty \left| \frac{e^{-t\lambda} - 1}{t^{1+\alpha}} \right|^2 \, d\mu_{f,f}^V(\lambda) \right)^{1/2} dt \\
            & \leq C \left( \int_0^1 \frac{dt}{t^{1+\alpha-\beta}} \left( \int_0^\infty \lambda^{2\beta} \, d\mu_{f,f}^V(\lambda) \right)^{1/2} + \int_1^\infty \frac{dt}{t^{1+\alpha}} \left(\int_0^\infty d\mu_{f,f}^V(\lambda)\right)^{1/2} \right) < \infty,
        \end{split}
    \end{equation*}
    because $\mu_{f,f}^V$ is a bounded measure and $f \in D(\mathcal{L}_V^\beta)$. 

    We conclude that $F_\alpha$ is $L^2(\mathbb{R}^d)$-Bochner integrable on $((0,\infty), dt)$. It follows that $f \in D(\mathcal{L}_{V,\textup{heat}}^\alpha)$ and that 
    \[
    \mathcal{L}_V^\alpha f = \mathcal{L}_{V,\textup{heat}}^\alpha f = \frac{1}{\Gamma(-\alpha)} \int_0^\infty \frac{T_t^V f - f}{t^{1+\alpha}} \, dt,
    \]
    which completes the proof.
\end{proof}

We define, for every $\alpha \in (0,1)$, the operator 
\[ 
\mathcal{L}_{V,\textup{heat}}^{-\alpha} f = \frac{1}{\Gamma(\alpha)} \lim_{m \rightarrow \infty} \int_{1/m}^m \frac{T_t^V f}{t^{1-\alpha}} \, dt, \quad f \in D(\mathcal{L}_{V,\textup{heat}}^{-\alpha}), 
\]
where the integrals are understood in the $L^2(\mathbb{R}^d)$-Bochner sense and the limit is taken in $L^2(\mathbb{R}^d)$. Here
\begin{equation*}
\begin{split}
    D(\mathcal{L}_{V,\textup{heat}}^{-\alpha}) = & \left\{ f \in L^2(\mathbb{R}^d) : \frac{T_t^V f}{t^{1-\alpha}} \text{ is } L^2(\mathbb{R}^d)\text{-Bochner integrable on } ((1/m,m), dt),\right. \\
    & \left. \text{ for every } m \in \mathbb{N}, \text{ and } \lim_{m \rightarrow \infty} \int_{1/m}^m \frac{T_t^V f}{t^{1-\alpha}} \, dt \text{ exists in } L^2(\mathbb{R}^d) \right\}.
\end{split}
\end{equation*}

By arguing as in Proposition~\ref{propo: 2.4} we can obtain the following result. 

\begin{prop}\label{propo: 2.5}
      Let $f \in L^2(\mathbb{R}^d)$ and $\alpha \in (0,1)$.
    \begin{enumerate}
        \item\label{item: propo 2.5 - a} If $f \in D(\mathcal{L}_V^{-\alpha})$, then $f \in D(\mathcal{L}_{V,\textup{heat}}^{-\alpha})$ and $\mathcal{L}_{V,\textup{heat}}^{-\alpha} f = \mathcal{L}_V^{-\alpha} f$.
        \item\label{item: propo 2.5 - b} If $f \in D(\mathcal{L}_V^{-\beta})$ for some $\alpha < \beta < 1$, then
        \[
        \mathcal{L}_V^{-\alpha} f = \frac{1}{\Gamma(\alpha)} \int_0^\infty \frac{T_t^V f}{t^{1-\alpha}} \, dt,
        \]
        where the integral converges in the $L^2(\mathbb{R}^d)$-Bochner sense.
    \end{enumerate}
\end{prop}

\begin{rem}
    Note that the properties established in this section can also be proved for a more general class of operators. As far as we know, these properties have not been studied carefully enough in the literature.
\end{rem}

\section{Pointwise representation of \texorpdfstring{$\log \mathcal{L}_V$}{log LV}}\label{sec: 3}

In this section we are going to prove Theorem~\ref{thm: 1.2}. 

Let $f\in C_c^\infty (\RR^d)$. According to Theorem~\ref{thm: 1.1} we have that
\[(\log \mathcal{L}_V)f=\lim_{s\to 0^+}\frac{\mathcal{L}_V^s f-f}{s}, \quad \text{in }L^2(\RR^d).\]
Then, for a certain sequence $\{s_n\}_{n\in \NN}\subseteq (0,\infty)$ such that $\lim_{n\to \infty}s_n=0$, we can write
\begin{equation}\label{eq: log LVsf ctp}
    (\log \mathcal{L}_V)f(x)=\lim_{n\to\infty}\frac{\mathcal{L}_V^{s_n} f(x)-f(x)}{s_n}, \quad \text{a.e. }x\in \RR^d.
\end{equation}
Since $f\in D(\mathcal{L}_V)$, it follows that 
\[\mathcal{L}_V^s f=\frac{1}{\Gamma(-s)}\int_0^\infty \frac{T_t^V(f)-f}{t^{s+1}}dt, \quad 0<s<1,\]
where the integral is understood in the $L^2(\RR^d)$-Bochner sense.

Indeed, let $g\in L^2(\RR^d)$. From \eqref{eq: log LVsf ctp} and the properties of the Bochner integrals, we can write
\begin{align*}
    \int_{\RR^d} g(x) \left(\int_0^\infty \frac{T_t^V(f)-f}{t^{s+1}}dt\right)(x) dx&=\int_0^\infty \int_{\RR^d} \frac{T_t^V(f)(x)-f(x)}{t^{s+1}}g(x) dx dt\\
    &=\int_0^1 \frac{T_t^V(f)(x)-f(x)}{t^{s+1}}g(x) dx dt+\int_1^\infty\frac{T_t^V(f)(x)-f(x)}{t^{s+1}}g(x) dx dt.
\end{align*}
Since $\sup_{t>0}\|T_t^V\|_{L^2(\RR^d)\to L^2(\RR^d)}<\infty$, we obtain
\begin{align*}
    \int_1^\infty\frac{\left|T_t^V(f)(x)-f(x)\right|}{t^{s+1}}&|g(x)| dx dt\\
    &\leq  \left(\sup_{t>0}\|T_t^Vf\|_{L^2(\RR^d)}+\|f\|_{L^2(\RR^d)}\right)\|g\|_{L^2(\RR^d)} \int_1^\infty \frac{dt}{t^{s+1}}\\
    &\leq \frac{C}{s} \|f\|_{L^2(\RR^d)}\|g\|_{L^2(\RR^d)}<\infty.
\end{align*}

On the other hand, since $f\in D(\mathcal L_V)$, there exists $C>0$ such that 
\[\|T_t^V(f)-f\|_{L^2(\RR^d)}\leq Ct, \quad t\in (0,1).\]
Then,
\[\int_0^1\int_{\mathbb R^d} \frac{|T_t^V(f)(x)-f(x)|}{t^{s+1}}|g(x)|dx dt\leq C \|g\|_{L^2(\RR^d)}\int_0^1 \frac{1}{t^s} dt <\infty.\]
Fubini's theorem leads to
\[\int_{\RR^d} \left(\int_0^\infty \frac{T_t^V(f)-f}{t^{s+1}}dt\right)(x) dx=\int_{\RR^d} g(x) \int_0^\infty \frac{T_t^V(f)(x)-f(x)}{t^{s+1}}dt dx.\]
The arbitrariness of $g$ implies that
\[\left(\int_0^\infty \frac{T_t^V(f)-f}{t^{s+1}}dt\right)(x)=\int_0^\infty \frac{T_t^V(f)(x)-f(x)}{t^{s+1}}dt, \quad \text{a.e. }x\in \RR^d.\]

We are going to see that, for every $x\in \RR^d$ and $r>0$,
\begin{equation}\label{eq: 2.1}
    \lim_{s\to 0^+} H_s(x)=\int_{B(x,r)} (f(x)-f(y))\int_0^\infty \frac{T_t^V(x,y)}{t}dtdy-\int_{\RR^d\setminus B(x,r)}f(y)\int_0^\infty \frac{T_t^V(x,y)}{t}dtdy-K(x,r)f(x),
\end{equation}
where
\begin{align*}
    K(x,r)&=2\log(\rho(x))+\int_0^{\rho^2(x)} \int_{\RR^d} \frac{T_t^V(x,y)-T_t(x-y)}{t}dydt-\int_0^{\rho^2(x)} \int_{\RR^d\setminus B(x,r)} \frac{T_t^V(x,y)}{t}dydt\\
    &\quad +\int_{\rho^2(x)}^\infty \int_{B(x,r)} \frac{T_t^V(x,y)}{t}dydt-\gamma,
\end{align*}
and 
\[H_s(x)=\frac1s \left(\frac{1}{\Gamma(-s)}\int_0^\infty \frac{T_t^V(f)(x)-f(x)}{t^{s+1}}dt-f(x)\right), \quad s\in (0,\tfrac12).\]

Fix $x\in \RR^d$. We can write
\begin{align*}
    \int_0^\infty \frac{T_t^V(f)(x)-f(x)}{t^{s+1}}dt&=\int_0^\infty \frac{T_t^V(f)(x)-T_ t^V(1)(x)f(x)}{t^{s+1}}dt+f(x)\int_0^\infty \frac{T_t^V(1)(x)-1}{t^{s+1}}dt\\
    &:=J_1(s)+J_2(s), \quad s\in (0,\tfrac12).
\end{align*}

According to \eqref{eq: 1.2} we have that
\begin{align*}
    \int_0^\infty \int_{\RR^d} \frac{T_t^V(x,y)}{t^{s+1}} |f(x)-f(y)| dy dt &\leq C \int_{\RR^d} |f(x)-f(y)|\int_0^\infty \frac{e^{-c|x-y|^2/t}}{t^{\frac d2 +s+1}}dt dy\\
    &\leq C\int_{\RR^d} \frac{|f(x)-f(y)|}{|x-y|^{d+2s}}dy\\
    &\leq C \left(\int_{B(x,1)} \frac{dy}{|x-y|^{d+2s-1}}+\int_{B^c(x,1)}\frac{dy}{|x-y|^{d+2s}}\right)\\
    &\leq C\left(\int_0^1 \frac{du}{u^{2s}}+\int_1^\infty \frac{du}{u^{2s+1}}\right)<\infty, \quad s\in (0,\tfrac12).
\end{align*}

Let $r>0$. It follows that
\begin{align*}
    J_1(s)&=\int_{\RR^d} (f(y)-f(x))\int_0^\infty \frac{T_t^V(x,y)}{t^{s+1}}dtdy\\
    &=\int_{B(x,r)} (f(y)-f(x))\int_0^\infty \frac{T_t^V(x,y)}{t^{s+1}}dtdy+\int_{\RR^d\setminus B(x,r)} f(y)\int_0^\infty \frac{T_t^V(x,y)}{t^{s+1}}dtdy\\
    &\quad -f(x)\int_{\RR^d\setminus B(x,r)}\int_0^\infty \frac{T_t^V(x,y)}{t^{s+1}}dtdy\\
    &:=J_{1,1}(s)+J_{1,2}(s)+J_{1,3}(s), \quad s\in (0,\tfrac12).
\end{align*}

Since $s\Gamma(-s)=-\Gamma(1-s)$, for $s\in(0,1)$, $\Gamma(-s)\sim -\frac1s$ as $s\to 0^+$. We will use this property several times in the sequel.

By proceeding as above, we get
\[\int_0^\infty \frac{T_t^V(x,y)}{t^{s+1}}dt\leq C \int_0^\infty \frac{e^{-c|x-y|^2/t}}{t^{\frac d2 +s+1}}dt \leq \frac{C}{|x-y|^{d+2s}}, \quad y\neq x.\]

If $|x-y|<r$, then $|x-y|^{d+2s}=r^{d+2s}\left|\frac{x-y}{r}\right|^{d+2s}\geq C|x-y|^{d+\frac12}$ for $s\in (0,1/4)$. 
Thus, since $f\in C_c^\infty(\RR^d)$, we obtain
\[|f(x)-f(y)|\int_0^\infty\frac{T_t^V(x,y)}{t^{s+1}}dt\leq \frac{C}{|x-y|^{d-\frac12}}, \quad 0<|x-y|<r.\]

Since $\int_{B(x,r)} |x-y|^{-d+\frac12}dy<\infty$, the Dominated Convergence Theorem implies that
\[\lim_{s\to 0^+} J_{1,1}(s)=\int_{B(x,r)} (f(y)-f(x)) \lim_{s\to 0^+} \int_0^\infty \frac{T_t^V(x,y)}{t^{s+1}}dt dy.\]
On the other hand, using again \eqref{eq: 1.2}, we get for every $x\neq y$ and $s\in (0,1)$ that
\[\frac{T_t^V(x,y)}{t^{s+1}}\leq C\begin{dcases*} \frac{e^{-c|x-y|^2/t}}{t^{\frac d2 +\frac32}} & $t\in (0,1)$\\ 
\frac{1}{t^{\frac d2 +1}} & $t\in (1,\infty).$
\end{dcases*}\]
Therefore, the Dominated Convergence Theorem leads to
\[\lim_{s\to 0^+} \int_0^\infty \frac{T_t^V(x,y)}{t^{s+1}}dt=\int_0^\infty \frac{T_t^V(x,y)}{t}dt,\]
so we conclude that
\[\lim_{s\to 0^+} \frac{1}{\Gamma(-s)s} J_{1,1}(s)=-\int_{B(x,r)} (f(y)-f(x))\int_0^\infty \frac{T_t^V(x,y)}{t}dt.\]

Since $\supp(f)$ is compact, by proceeding as before we obtain that
\[\lim_{s\to 0^+} \frac{1}{\Gamma(-s)s} J_{1,2}(s)=-\int_{\RR^d\setminus B(x,r)} f(y)\int_0^\infty \frac{T_t^V(x,y)}{t}dt.\]

We postpone, for the moment, the study of $J_{1,3}$ and proceed to estimate $J_2$.

We decompose $J_2$ as follows,
\begin{align*}
    J_2(s)&=f(x)\int_0^\infty \frac{T_t^V(1)(x)-T_t(1)(x)}{t^{s+1}}dt\\
    &=f(x)\int_0^{\rho^2(x)} \int_{\RR^d} \frac{T_t^V(x,y)-T_t(x-y)}{t^{s+1}}dydt +f(x)\int_{\rho^2(x)}^\infty \int_{\RR^d\setminus B(x,r)} \frac{T_t^V(x,y)}{t^{s+1}}dydt\\
    &\quad +{f(x)}\int_{\rho^2(x)}^\infty \int_{B(x,r)} \frac{T_t^V(x,y)}{t^{s+1}}dydt-f(x)\int_{\rho^2(x)}^\infty \int_{\RR^d} \frac{T_t(x-y)}{t^{s+1}}dydt\\
    &:=J_{2,1}(s)+J_{2,2}(s)+J_{2,3}(s)+J_{2,4}(s), 
\end{align*}
for $s\in \left(0,\frac{\delta}{4}\right)$, where $\delta>0$ is as in \eqref{eq: 1.5}.

Clearly, 
\[J_{2,4}(s)=-\frac{\rho(x)^{-2s}}{s}f(x), \quad s\in \left(0,\tfrac12\right).\]
Then,
\begin{align*}
    \lim_{s\to 0^+} \frac1s\left(\frac{1}{\Gamma(-s)}J_{2,4}(s)-f(x)\right)&=f(x)\lim_{s\to 0^+} \left(\frac{\rho(x)^{-2s}-1}{s}+\frac{1-\Gamma(1-s)}{s\Gamma(1-s)}\rho(x)^{-2s}\right)\\
    &=f(x)\left(-2\log(\rho(x))+\Gamma'(1)\right)=f(x)\left(-2\log(\rho(x))-\gamma\right).
\end{align*}

To estimate $J_{2,1}$, we consider $0<s<\frac{\delta}{4}$. According to \eqref{eq: 1.5}, for any $0<t<\rho^2(x)$ and $y\in \RR^d$, we can obtain
\begin{align*}
    \frac{|T_t^V(x,y)-T_t(x-y)|}{t^{s+1}}&\leq  \left(\frac{\sqrt{t}}{\rho(x)}\right)^\delta \frac{\omega_t(x-y)}{t^{s+1}} =  \frac{\omega_t(x-y)}{\left(\frac{t}{\rho^2(x)}\right)^{s+1-\delta/2}(\rho(x))^{2s+2}}\\
    &\leq \frac{\omega_t(x-y)}{\left(\frac{t}{\rho^2(x)}\right)^{1-\delta/4}(\rho(x))^{2s+2}}\\
    &= (\rho(x))^{-2s-\delta/2} \frac{\omega_t(x-y)}{t^{1-\delta/4}}\\
    &\leq \max\left\{(\rho(x))^{-\delta}, (\rho(x))^{-\delta/2}\right\} \frac{\omega_t(x-y)}{t^{1-\delta/4}}.
\end{align*}
Since $\int_0^{\rho^2(x)}\int_{\RR^d} \frac{\omega_t(x-y)}{t^{1-\delta/4}} dydt<\infty$, it follows that
\[\lim_{s\to 0^+} \frac{1}{s\Gamma(-s)}J_{2,1}(s)=-f(x)\int_0^{\rho^2(x)}\int_{\RR^d} \frac{T_t^V(x,y)-T_t(x-y)}{t}dydt.\]

We are going to combine now the estimates for $J_{1,3}$ and $J_{2,2}$. By using again \eqref{eq: 1.2}, if $0<s<\frac{\delta}{4}$, $0<t<\rho^2(x)$ and $|x-y|>r$, we get
\begin{align*}
    \frac{T_t^V(x,y)}{t^{s+1}}&\leq C \frac{e^{-c|x-y|^2/t}}{t^{d/2+s+1}}\leq C\frac{e^{-\frac{c}{2t}r^2} e^{-\frac{c}{2t}|x-y|^2}}{\left(\frac{t}{\rho^2(x)}\right)^{d/2+s+1}\rho(x)^{d+2s+2}}\\
    &\leq C \max\left\{(\rho(x))^{-d-2},(\rho(x))^{-d-2-\delta/2}\right\} \frac{e^{-\frac{c}{2t}r^2} e^{-\frac{c}{2t}|x-y|^2}}{t^{d/2+\delta/4+1}}.
\end{align*}
Noting that
\[\int_0^{\rho^2(x)} \frac{e^{-\frac{c}{2t}r^2}}{t^{d/2+\delta/4+1}}\int_{\RR^d} e^{-\frac{c}{2t}|x-y|^2} dydt\leq C \int_0^{\rho^2(x)} \frac{e^{-\frac{c}{2t}r^2}{t^{d/2+s+1}}}{t^{\delta/4+1}}dt<\infty,\]
it yields
\[\lim_{s\to 0^+} \frac{1}{s\Gamma(-s)}\left(J_{1,3}(s)+J_{2,2}(s)\right)=f(x)\int_0^{\rho^2(x)}\int_{\RR^d \setminus B(x,r)} \frac{T_t^V(x,y)}{t}dy dt.\]

Finally, for $t>\rho^2(x)$ and $0<s<\delta/4$, \eqref{eq: 1.2} leads to
\[\frac{T_t^V(x,y)}{t^{s+1}}\leq C \frac{e^{-c|x-y|^2/t}}{t^{d/2+s+1}}\leq \frac{C}{\left(\frac{t}{\rho^2(x)}\right)^{d/2+s+1}\rho(x)^{d+2s+2}}\leq \frac{C}{t^{d/2+1}}.\]
Since $\int_{\rho^2(x)}^\infty \int_{B(x,r)} t^{-d/2-1}dt<\infty$, we have
\[\lim_{s\to 0^+} \frac{1}{s\Gamma(-s)} J_{2,3}(s)=-f(x)\int_{\rho^2(x)}^\infty \int_{B(x,r)}\frac{T_t^V(x,y)}{t} dydt.\]

By combining all of the above estimates, we conclude that \eqref{eq: teo 1.2} holds. By taking $r=1$ in \eqref{eq: teo 1.2} we have established part \ref{item: teo 1.2 - a} of Theorem~\ref{thm: 1.2}.

Let us now prove \ref{item: teo 1.2 - b}. We choose $r>1$ such that $\supp(f)\subseteq B(0,r)$.

First, assume that $|x|\geq 2r$, so $f(x)=0$. Then,
\[\int_0^\infty \frac{T_t^Vf(x)-f(x)}{t^{s+1}}dt=\int_{B(0,r)}f(y)\int_0^\infty \frac{T_t^V(x,y)}{t^{s+1}}dt dy.\]
Note that for every $y\in B(0,r)$, $|x-y|\geq |x|-|y|\geq \frac{|x|}{2}\geq r>1$. Then, by using \eqref{eq: 1.2}, we get
\begin{align*}
    \left|\int_0^\infty \frac{T_t^Vf(x)-f(x)}{t^{s+1}}dt\right|&\leq C\int_{B(0,r)}|f(y)|\int_0^\infty \frac{e^{-c\frac{|x-y|^2}{t}}}{t^{\frac{d}{2}+s+1}}dtdy\\
    &\leq C \int_{B(0,r)} \frac{|f(y)|}{|x-y|^{d+2s}}dy\\
    &\leq C \frac{\|f\|_{L^1(\RR^d)}}{|x|^d}, 
\end{align*}
for every $s\in (0,1)$. It follows that, for $1<p<\infty$
\begin{equation}\label{eq: bound Lp out B}
    \frac{1}{s|\Gamma(-s)|}\left\|\int_0^\infty\frac{T_t^Vf(\cdot)-f(\cdot)}{t^{s+1}}dt\right\|_{L^p(\RR^d\setminus B(0,2r))}\leq C \|f\|_{L^1(\RR^d)} r^{-\frac{d}{p'}},
\end{equation}
and
\begin{equation}\label{eq: bound Linfty out B}
    \frac{1}{s|\Gamma(-s)|}\left\|\int_0^\infty\frac{T_t^Vf(\cdot)-f(\cdot)}{t^{s+1}}dt\right\|_{L^\infty(\RR^d\setminus B(0,2r))}\leq C \|f\|_{L^1(\RR^d)} r^{-d},
\end{equation}
for every $s\in (0,1)$. Here, the constants are independent of $s$ since $s|\Gamma(-s)|\sim 1$ as $s\to 0^+$.

Suppose now $|x|<2r$. If $|x-y|\geq  3r$, then $|y|\geq |x-y|-|x|>r$. Therefore, $\supp(f) \subseteq B(x,3r)$ and we have
\begin{align*}
    \int_{\RR^d}(f(y)-f(x))\int_0^\infty \frac{T_t^V(x,y)}{t^{s+1}}dtdy&=\int_{B(x,3r)} (f(y)-f(x))\int_0^\infty \frac{T_t^V(x,y)}{t^{s+1}}dtdy\\
    &\quad -f(x) \int_{\RR^d\setminus B(x,3r)} \int_0^\infty \frac{T_t^V(x,y)}{t^{s+1}}dtdy\\
    &:=J_{1,1}(s,x)+J_{1,3}(s,x), \quad s\in (0,1).
\end{align*}
We will prove that
\begin{equation}\label{eq: lim J11}
    \frac{J_{1,1}(s,x)}{s\Gamma(-s)}\underset{s\to 0^+}{\longrightarrow} -\int_{B(x,3r)}(f(y)-f(x))\int_0^\infty \frac{T_t^V(x,y)}{t}dtdy
\end{equation}
uniformly on $x\in B(0,2r)$. Indeed, by using \eqref{eq: 1.2}, we get
\begin{align*}
    &\left| \frac{J_{1,1}(s,x)}{s\Gamma(-s)}+\int_{B(x,3r)}(f(y)-f(x))\int_0^\infty \frac{T_t^V(x,y)}{t}dtdy\right|\\
    &\leq \int_{B(x,3r)}|f(y)-f(x)| \int_0^\infty T_t^V(x,y)\left|\frac{1}{t^{s+1}}-\frac1t\right|dtdy\\
    &\leq C \int_{B(x,3r)}|x-y| \int_0^\infty \frac{e^{-c\frac{|x-y|^2}{t}}}{t^{\frac d2}}\left|\frac{1}{t^{s+1}}-\frac1t\right|dtdy\\
    &\leq C \int_0^{3r} u^d \int_0^\infty \frac{e^{-c\frac{u^2}{t}}}{t^{\frac d2}}\left|\frac{1}{t^{s+1}}-\frac1t\right|dtdu.
\end{align*}
We observe that, for every $t\in (0,\infty)$ and $s\in (0,\tfrac12)$,
\[\frac{e^{-c\frac{u^2}{t}}}{t^{\frac d2+1}}\left|\frac{1}{t^{s}}-1\right|\leq C\left(\frac{e^{-c\frac{u^2}{t}}}{t^{\frac d2+1}}\chi_{[1,\infty)}(t)+\frac{e^{-c\frac{u^2}{t}}}{t^{\frac d2+\frac32}}\chi_{(0,1)}(t)\right).\]
Hence, we have
\begin{align*}
    \int_0^{3r} u^d \int_0^\infty \frac{e^{-c\frac{u^2}{t}}}{t^{\frac d2}}\left|\frac{1}{t^{s+1}}-\frac1t\right|dtdu &\lesssim \int_0^{3r} u^d \int_0^1 \frac{e^{-c\frac{u^2}{t}}}{t^{\frac d2+1}} dtdu+\int_0^{3r} u^d \int_1^\infty \frac{e^{-c\frac{u^2}{t}}}{t^{\frac d2+\frac32}} dtdu\\
    &\lesssim \int_0^{3r} u^d \int_0^1 \left(\frac{t}{u^2}\right)^{\frac d2+\frac14} \frac{dt}{t^{\frac d2+1}}du+\int_0^{3r} u^d \int_1^\infty \frac{1}{t^{\frac d2+\frac32}} dtdu\\
    &\lesssim \left(\int_0^{3r} \frac{du}{u^{\frac12}} \right)\left(\int_0^1 \frac{dt}{t^{\frac34}}\right)+\int_0^{3r} u^d \int_1^\infty \frac{dt}{t^{\frac32}} du<\infty.
\end{align*}
The Dominated Convergence Theorem leads to
\[\lim_{s\to 0^+} \int_0^{3r} u^d \int_0^\infty \frac{e^{-c\frac{u^2}{t}}}{t^{\frac d2}}\left|\frac{1}{t^{s+1}}-\frac1t\right|dtdu=0.\]
Thus, \eqref{eq: lim J11} is proved.

As before, we proceed to estimate $J_2(s,x)$ given below, and leave the study of $J_{1,3}(s,x)$ for later. We write
\begin{align*}
    J_2(s,x)&=f(x)\int_0^\infty \frac{T_t^V(1)(x)-T_t(1)(x)}{t^{s+1}} dt\\
    &=f(x)\int_0^{\rho^2(x)} \int_{\RR^d} \frac{T_t^V(x,y)-T_t(x-y)}{t^{s+1}}dydt +f(x)\int_{\rho^2(x)}^\infty \int_{\RR^d\setminus B(x,r)} \frac{T_t^V(x,y)}{t^{s+1}}dydt\\
    &\quad +f(x)\int_{\rho^2(x)}^\infty \int_{B(x,r)} \frac{T_t^V(x,y)}{t^{s+1}}dydt-f(x)\int_{\rho^2(x)}^\infty \int_{\RR^d} \frac{T_t(x-y)}{t^{s+1}}dydt\\
    &:=J_{2,1}(s,x)+J_{2,2}(s,x)+J_{2,3}(s,x)+J_{2,4}(s,x)
\end{align*}
for $s\in (0,\frac{\delta}{4})$.

By \eqref{eq: 1.5} we get
\begin{align*}
    \left|\frac{J_{2,1}(s,x)}{s\Gamma(-s)}+\int_0^{\rho^2(x)} \int_{\RR^d} \frac{T_t^V(x,y)-T_t(x-y)}{t}dydt\right|&\leq \int_0^{\rho^2(x)} \int_{\RR^d} |T_t^V(x,y)-T_t(x-y)|\left|\frac{1}{t^{s+1}}-\frac1t\right|dydt\\
    &\lesssim \int_0^{\rho^2(x)}\left|\frac{1}{t^{s+1}}-\frac1t\right| \int_{\RR^d} \left(\frac{\sqrt{t}}{\rho(x)}\right)^\delta \omega_t(x-y) dydt\\
    &\lesssim \int_0^{\rho^2(x)} \left(\frac{\sqrt{t}}{\rho(x)}\right)^\delta \left|\frac{1}{t^{s}}-1\right|\frac{dt}{t},
\end{align*}
for every $s\in (0,\frac{\delta}{4})$.

Suppose $\mathbb K$ is some compact subset of $\RR^d$. Then, there exist $x_1,\dots, x_m\in \mathbb K$ such that $\mathbb K\subseteq \bigcup_{j=1}^m B(x_j,\rho(x_j))$. Therefore, from \eqref{eq: rox_vs_roy}, for every $y\in B(x_j,\rho(x_j))$ and $j=1, \dots, m$, we can obtain that there exist constants $0<K_1<K_2<\infty$ for which $K_1\leq \rho(y)\leq K_2$ for every $y\in \mathbb K$.

Applying this property on $\overline{B(0,2r)}$, we get that, for every $s\in (0,\frac{\delta}{4})$,
\begin{equation*}
    \left|\frac{J_{2,1}(s,x)}{s\Gamma(-s)}+\int_0^{\rho^2(x)} \int_{\RR^d} \frac{T_t^V(x,y)-T_t(x-y)}{t}dydt\right|\lesssim \frac{1}{K_1^\delta} \int_0^{A} t^{\frac{\delta}{2}-1}\left|\frac{1}{t^{s}}-1\right|dt
\end{equation*}
for $A=K_2^2>0$.

Since 
\begin{align}\label{eq: integral 0 a A}
    \int_0^{A} t^{\frac{\delta}{2}-1}\left|\frac{1}{t^{s}}-1\right|dt& \leq \int_0^{A} t^{\frac{\delta}{2}-1}\left(\frac{1}{t^{s}}+1\right)dt\leq \int_0^{A} t^{\frac{\delta}{2}-1}+\frac{1}{A^s}\int_0^{A} \frac{t^{\frac{\delta}{2}-1}}{\left(\frac{t}{A}\right)^{s}}dt\nonumber\\
    &\leq \int_0^{A} t^{\frac{\delta}{2}-1}+A^{\frac{\delta}{4}-s}\int_0^{A} t^{\frac{\delta}{4}-1} dt <\infty
\end{align}
for every $s\in (0,\frac{\delta}{4})$, we conclude that
\begin{equation}\label{eq: lim J21}
    \frac{J_{2,1}(s,x)}{s\Gamma(-s)}\underset{s\to 0^+}{\longrightarrow} -\int_0^{\rho^2(x)} \int_{\RR^d} \frac{T_t^V(x,y)-T_t(x-y)}{t}dydt
\end{equation}
uniformly in $x\in B(0,2r)$.

We also have that 
\begin{align*}
    \left|\frac{J_{1,3}(s,x)+J_{2,2}(s,x)}{s\Gamma(-s)}+\int_0^{\rho^2(x)}\int_{\RR^d\setminus B(x,3r)}\frac{T_t^V(x,y)}{t}dydt\right|&\lesssim \int_0^{\rho^2(x)}\int_{\RR^d\setminus B(x,3r)}T_t^V(x,y)\left|\frac{1}{t^{s+1}}-\frac1t\right|dydt\\
    &\lesssim \int_0^{\rho^2(x)}\int_{\RR^d\setminus B(x,3r)}\frac{e^{-c\frac{|x-y|^2}{t}}}{t^{\frac d2}}\left|\frac{1}{t^{s+1}}-\frac1t\right|dydt\\
    &\lesssim \int_0^{\rho^2(x)}\frac{e^{-c\frac{r^2}{t}}}{t}\left|\frac{1}{t^{s}}-1\right|dt\\
    &\lesssim r^{-\delta}\int_0^A t^{\frac{\delta}{2}-1}\left|\frac{1}{t^{s}}-1\right|dt<\infty
\end{align*}
proceeding as in \eqref{eq: integral 0 a A}. Therefore,
\begin{equation}\label{eq: lim J13+J22}
    \frac{J_{1,3}(s,x)+J_{2,2}(s,x)}{s\Gamma(-s)}\underset{s\to 0^+}{\longrightarrow}-\int_0^{\rho^2(x)}\int_{\RR^d\setminus B(x,3r)}\frac{T_t^V(x,y)}{t}dydt
\end{equation}
uniformly in $x\in B(0,2r)$.

To estimate the limit for $J_{2,3}(s,x)$ as $s\to 0^+$, we can write
\begin{align*}
    \left|\frac{1}{s\Gamma(-s)} J_{2,3}(s,x)+\int_{\rho^2(x)}^\infty\int_{B(x,3r)}\frac{T_t^V(x,y)}{t}dydt\right|&\lesssim \int_0^{\rho^2(x)}\int_{B(x,3r)}\frac{e^{-c\frac{|x-y|^2}{t}}}{t^{\frac d2}}\left|\frac{1}{t^{s+1}}-\frac1t\right|dydt\\
    &\lesssim r^d\int_A^\infty \frac{1}{t^{\frac d2+1}}\left|\frac{1}{t^s}-1\right|dt, 
\end{align*}
where, as before, the constants involved do not depend on $x$. Since
\[ \frac{1}{t^{\frac d2+1}}\left(\frac{1}{t^s}+1\right)= \frac{1}{t^{\frac d2+1}}\left(\frac{1}{A^s\left(\frac{t}{A}\right)^s}+1\right)\leq C(A) \frac{1}{t^{\frac d2+1}}\]
for every $s\in (0,1)$ and $t>A$. Hence,
\[\lim_{s\to 0^+} \int_A^\infty \frac{1}{t^{\frac d2+1}}\left|\frac{1}{t^s}-1\right|dt=0\]
and
\begin{equation}\label{eq: lim J23}
    \frac{1}{s\Gamma(-s)} J_{2,3}(s,x)\underset{s\to 0^+}{\longrightarrow} -\int_{\rho^2(x)}^\infty\int_{B(x,3r)}\frac{T_t^V(x,y)}{t}dydt
\end{equation}
uniformly in $x\in B(0,2r)$.

Finally, it is left to prove
\begin{equation}\label{eq: lim J24}
    \frac{\frac{1}{\Gamma(-s)} J_{2,4}(s,x)-1}{s}\underset{s\to 0^+}{\longrightarrow} -2\log(\rho(x))-\gamma
\end{equation}
uniformly in $x\in B(0,2r)$. 

By using the differential mean value theorem, we have
\[\frac{u^{-2s}-1}{u}=-2\log (u) u^{-2\xi}\]
for some $\xi\in (0,s)$, with $s\in (0,1)$ and $u>0$. Let $0<a<b<\infty$. We have that
\[b^{-2z}-1\leq u^{-2z}-1\leq a^{-2z}-1, \quad z>0, \ a\leq u\leq b.\]
Hence, $u^{-2z}\to 1$ as $z\to 0^+$ uniformly in $u\in[a,b]$. This yields $\frac{u^{-2s}-1}{u}\to -2\log(u)$ as $s\to 0^+$, uniformly in $u\in[a,b]$. 

As it was proved above, since $K_1\leq \rho(x)\leq K_2$ for every $x\in B(0,2r)$, we conclude that \eqref{eq: lim J24} holds.

By combining \eqref{eq: lim J11}, \eqref{eq: lim J21}, \eqref{eq: lim J13+J22}, \eqref{eq: lim J23} and \eqref{eq: lim J24}, we obtain that
\begin{align}\label{eq: lim uniform}
    \frac1s \left(\frac{1}{\Gamma(-s)}\int_0^\infty \frac{T_t^V(f)(x)-f(x)}{t^{s+1}}dt-1\right)\underset{s\to 0^+}{\longrightarrow}& \int_{B(x,r)}(f(x)-f(y))\int_0^\infty \frac{T_t^V(x,y)}{t}dtdy\nonumber\\
    & -\int_{\RR^d\setminus B(x,3r)} f(y)\int_0^\infty \frac{T_t^V(x,y)}{t}dtdy-K(x,3r)f(x)
\end{align}
uniformly in $x\in B(0,2r)$.

Let us prove the convergence on $L^p(\RR^d)$ for $1<p\leq \infty$. Let $\epsilon>0$. According to \eqref{eq: bound Lp out B} and \eqref{eq: bound Linfty out B}, there exists $r_0>0$ such that
\[\frac{1}{s|\Gamma(-s)|}\left\|\int_0^\infty \frac{T_t^V(f)(\cdot)-f(\cdot)}{t^{s+1}}dt\right\|_{L^p(\RR^d\setminus B(0,2r_0))}<\epsilon, \quad s\in (0,1).\]

By \eqref{eq: lim uniform}, there exists $s_0\in (0,1)$ for which
\begin{align*}
    &\left\|\frac1s \left(\frac{1}{\Gamma(-s)}\int_0^\infty \frac{T_t^V(f)(\cdot)-f(\cdot)}{t^{s+1}}dt-1\right)-\int_{B(\cdot,r)}(f(\cdot)-f(y))\int_0^\infty \frac{T_t^V(\cdot,y)}{t}dtdy\right. \\
    & \quad \left.-\int_{\RR^d\setminus B(\cdot,3r)} f(y)\int_0^\infty \frac{T_t^V(\cdot,y)}{t}dtdy-K(\cdot,3r)f(\cdot)\right\|_{L^p(B(0,2r_0))}<\epsilon, \quad 0<s<s_0.
\end{align*}
Therefore,
\[\lim_{s\to 0^+} \frac1s \left(\frac{1}{\Gamma(-s)}\int_0^\infty \frac{T_t^V(f)-f}{t^{s+1}}dt-1\right)=(\log \mathcal{L}_V)f, \quad \text{in }L^p(\RR^d),\]
and Theorem~\ref{thm: 1.2} is now proved.

The results in Theorem~\ref{thm: 1.2} are also valid when we take the limit on $s$ to zero from the left.

\begin{cor}\label{coro: 2.1}
    Let $d\geq 3$ and $V\in \RH_q$ with $q>\frac d2$. If $f\in \Lip^\theta (\RR^d)$ for some $\theta\in (0,1]$ and $\int_{\RR^d} |f(y)|(1+|y|)^{-d}dy<\infty$, then
    \[(\log \mathcal{L}_V)(f)(x)=-\lim_{h\to 0^+}\frac{\mathcal{L}^{-h}_V-I}{h}(f)(x), \quad  x\in \RR^d.\]
\end{cor}

\begin{proof}Let $h>0$ and $x\in \RR^d$. We can write
\begin{align*}
    \mathcal{L}^{-h}_Vf(x)&-f(x)\\
    &=\frac{1}{\Gamma(h)}\int_0^{\rho^2(x)}\left(T_t^V(f)(x)-T_t^V(1)(x)f(x)\right)t^{h-1} dt+\frac{1}{\Gamma(h)}\int_{\rho^2(x)}^\infty T_t^V(f)(x)t^{h-1} dt\\
    &\quad +\frac{f(x)}{\Gamma(h)}\int_0^{\rho^2(x)} \left(T_t^V(1)(x)-T_t(1)(x)\right)t^{h-1}dt+\left(\frac{1}{\Gamma(h)}\int_0^{\rho^2(x)} T_t(1)(x)t^{h-1} dt-1\right)f(x)\\
    &:=A_1(x,h)+A_2(x,h)+A_3(x,h)+A_4(x,h).
\end{align*}

By using that $f\in \Lip^\theta(\RR^d)$ and (\ref{eq: 1.2}) we obtain
\begin{align*}
 T_t^V(x,y)|f(y)-f(x)|t^{h-1}&\le C|x-y|^\theta\frac{e^{-c\frac{|x-y|^2}{t}}}{t^{d/2+1-h}}\\
 &\le C\frac{e^{-\frac c2\frac{|x-y|^2}{t}}}{t^{d/2}}\left(\frac{\rho^2(x)}{t}\right)^{1-h-\theta/2}\rho(x)^{\theta-2+2h}\\
 &\le C\frac{e^{-\frac c2\frac{|x-y|^2}{t}}}{t^{d/2}}\left(\frac{\rho^2(x)}{t}\right)^{1-\theta/2}\max\{\rho^{\theta-2}(x),\rho^{\theta}(x)\}\\
 &\le C\frac{e^{-\frac c2\frac{|x-y|^2}{t}}}{t^{1+(d-\theta)/2}}\max\{1,\rho^2(x)\},
    \end{align*}
when $h\in (0,1)$, $y\in\mathbb{R}^d$ and $0<t<\rho^2(x)$.

Since 
\[
\int_0^{\rho^2(x)}\int_{\mathbb{R}^d}\frac{e^{-\frac c2\frac{|x-y|^2}{t}}}{t^{(d-\theta)/2+1}}dydt\le C\int_0^{\rho^2(x)}t^{\theta/2-1}dt\le C\rho^\theta(x),
\]
the Dominated Convergence Theorem leads to
\[
\lim_{h\to 0^+}\frac{A_1(x,h)}{h}=\int_0^{\rho^2(x)}\frac{T_t^V(f)(x)-T_t^V(1)(x)f(x)}{t}dt.
\]


By using that $f\in \Lip^\theta(\mathbb{R}^d)$ and (\ref{eq: 1.3}) we get
\begin{align*}
T_t^V(x,y)|f(y)|t^{h-1}&\le {C_N}\frac{e^{-c\frac{|x-y|^2}{t}}}{t^{1-h+d/2}} \left(\frac{\rho(x)}{\rho(x)+\sqrt{t}}\right)^{{N}}(|x-y|^\theta+|f(x)|)\\
&\le {C_N}\rho(x)^{{N}}\frac{e^{-\frac c2\frac{|x-y|^2}{t}}}{t^{{N/2+1}-h+d/2}}(t^{\theta/2}+|f(x)|)\\
&\le {C_N}\rho(x)^{{N}}\frac{e^{-\frac c2\frac{|x-y|^2}{t}}}{t^{d/2}}\left(\frac{\rho^2(x)}{t}\right)^{{N/2+1}-h}(t^{\theta/2}+|f(x)|)\rho(x)^{2h-{N+2}}\\
&\le {C_N}\rho(x)^{{N}}\frac{e^{-\frac c2\frac{|x-y|^2}{t}}}{t^{d/2}}\left(\frac{\rho^2(x)}{t}\right)^{{N/2}}(t^{\theta/2}+|f(x)|)\rho(x)^{2h-{N+2}}\\
&\le {C_N}\rho(x)^{{N-2}+2h}\frac{e^{-\frac c2\frac{|x-y|^2}{t}}}{t^{d/2}}\frac{1}{t^{{N/2}}}(t^{\theta/2}+|f(x)|)\\
\end{align*}
for $h\in (0,1)$, $y\in \mathbb{R}^d$ and $t\ge \rho^2(x)$.

{By choosing $N>3$,}
\[
\int_{\rho^2(x)}^\infty \int_{\mathbb{R}^d}\frac{e^{-\frac c2\frac{|x-y|^2}{t}}}{t^{d/2}}\frac{1}{t^{{N/2}}}(t^{\theta/2}+|f(x)|)dydt\le C\int_{\rho^2(x)}^\infty\frac{t^{\theta/2}+|f(x)|}{t^{{N/2}}}dt<\infty.
\]
According to Dominated Convergence Theorem we get
\[
\lim_{h\to 0^+}\frac{A_2(x,h)}{h}=\int_{\rho^2(x)}^\infty\frac{T_t^V(f)(x)}{t}dt.
\]


By (\ref{eq: 1.5}) it follows that
\begin{align*}
 |T_t^V(x,y)-T_t(x-y)|t^{h-1}&\le C\frac{1}{\rho(x)^\delta}\frac{\omega_t(x-y)}{t^{1-h-\delta/2}}\\
& \le C\rho(x)^{2h-2} \omega_t(x-y) \left(\frac{\rho^2(x)}{t}\right)^{1-h-\delta/2}\\
&\le C\rho(x)^{2h-\delta}\omega_t(x-y) t^{-1+\delta/2}\\
&\le C\max\{\rho(x)^{2-\delta},\rho(x)^{-\delta}\}\omega_t(x-y)t^{-1+\delta/2},
\end{align*}
when $h\in (0,1)$, $y\in \mathbb{R}^d$ and $0<t<\rho^2(x)$.

Since
\[
\int_0^{\rho^2(x)}\int_{\mathbb{R}^d}\omega_t(x-y)t^{-1+\delta/2}dydt\le C\rho(x)^\delta,
\]
the Dominated Convergence Theorem leads to
\[
\lim_{h\to 0^+}\frac{A_3(x,h)}{h}=f(x)\int_0^{\rho^2(x)}\frac{T_t^V(1)(x)-T_t(1)(x)}{t}dt.
\]


Finally, we have that
\begin{align*}
\frac{A_4(x,h)}{h}&=\frac{f(x)}{h}\left(\frac{\rho(x)^{2h}}{h\Gamma(h)}-1\right)=f(x)\left(\frac{\rho(x)^{2h}-1}{\Gamma(h+1)h}+ \frac{1-\Gamma(h+1)}{\Gamma(h+1)h}\right), \quad h\in (0,1).
\end{align*}
Then,
\[
\lim_{h\to 0^+}\frac{A_4(x,h)}{h}=f(x)( \gamma+2\log\rho(x)).
\]


We conclude that 
\begin{align*}
-\lim_{h\to 0^+}\frac{\mathcal{L}_V^{-h}(f)(x)-f(x)}{h}&=\int_0^{\rho^2(x)}\frac{T_t^V(1)f(x)-T_t^V(f)(x)}{t}dt\\
&-\int_{\rho^2(x)}^\infty \frac{T_t^V(f)(x)}{t}dt\\
&+f(x)\int_0^{\rho^2(x)}\frac{T_t(1)(x)-T_t^V(1)(x)}{t}dt\\
&-f(x)(\gamma+2\log\rho(x)).
\end{align*}

    By a rearrangement of the terms and taking into account that $\int_{\mathbb{R}^d}|f(y)|(1+|y|)^{-d}<\infty$,  Theorem~\ref{thm: 1.2} and the comments after this one allow us to conclude that
    \[-\lim_{h\to 0^+}\frac{\mathcal{L}^{-h}_V(f)(x)-f(x)}{h}=\lim_{h\to 0^+}\frac{\mathcal{L}^{h}_V(f)(x)-f(x)}{h}=(\log \mathcal{L}_V)(f)(x).\qedhere\]
\end{proof}

\section{Proof of Theorem~\ref{thm: 1.3}}\label{sec: 4}

We define $u(x,t) = \mathcal{L}_V^{-t} f(x)$, $x \in \mathbb{R}^d$ and $t > 0$. Since $f \in L^1(\mathbb{R}^d)$, for every $0 < t < d/2$, $\mathcal{L}_V^{-t} f \in L^{d/(d-2t),\infty}(\mathbb{R}^d)$. Moreover, since $f \in \Lip_V^\theta$, we have $\mathcal{L}_V^{-t} f \in \Lip_V^{\theta + t}$ for $t \in (0, 1-\theta)$. Then, $\mathcal{L}_V^{-t} f \in \Lip^{\theta + t}$ and $\| \mathcal{L}_V^{-t} f (\cdot) \rho(\cdot)^{-(\theta+t)} \|_\infty < \infty$ for every $t \in (0, 1-\theta)$. It follows that

\begin{equation*}
    \int_{\mathbb{R}^d} \frac{|\mathcal{L}_V^{-t} f(x)|}{(1+|x|)^d} \, dx \leq \| \mathcal{L}_V^{-t} f(\cdot) \rho(\cdot)^{-(\theta+t)} \|_{\infty}
    \int_{\mathbb{R}^d} \frac{\rho(x)^{\theta+t}}{(1+|x|)^d} \, dx, \quad t \in (0, 1-\theta).
\end{equation*}

According to the comments after Theorem~\ref{thm: 1.2} and Corollary~\ref{coro: 2.1}, we obtain
\begin{equation}\label{eq: 3.1}
    \Log (\mathcal{L}_V^{-t} f)(x) = \lim_{h \rightarrow 0^+} \frac{\mathcal{L}_V^{-h} - I}{h} \mathcal{L}_V^{-t} f(x), \quad x \in \mathbb{R}^d \text{ and } t \in (0, 1-\theta).
\end{equation}
We are going to show that
\begin{equation*}
    \mathcal{L}_V^{-h} (\mathcal{L}_V^{-t} f)(x) = \mathcal{L}_V^{-(h+t)} f(x), \quad t > 0, \, x \in \mathbb{R}^d \text{ and } 0 < h < 1 - t - \theta.
\end{equation*}

Assume first that $0 < h + t + \theta < 1$, with $h, t > 0$. We have that
\begin{equation*}
    \mathcal{L}_V^{-t} f(x) = \frac{1}{\Gamma(t)} \int_{\mathbb{R}^d} f(y) \int_0^\infty T_u^V(x,y) u^{t-1} \, du \, dy, \quad x \in \mathbb{R}^d.
\end{equation*}

We can write
\begin{equation*}
    \begin{split}
        \int_{\mathbb{R}^d} \int_{\mathbb{R}^d} |f(y)| & \int_{0}^\infty T_u^V(z,y) u^{t-1} \, du \int_0^{\infty} T_s^V(x,z) s^{h-1} \, ds \, dy \, dz \\
        & = \int_{\mathbb{R}^d} |f(y)| \int_{0}^\infty \int_0^{\infty} \int_{\mathbb{R}^d} T_u^V(z,y) T_s^V(x,z) \, dz \, u^{t-1} s^{h-1} \, du \, ds \, dy \\
        & = \int_{\mathbb{R}^d} |f(y)| \int_{0}^\infty \int_0^{\infty} T_{u+s}^V(x,y) u^{t-1} s^{h-1} \, du \, ds \, dy \\
        & = \int_{\mathbb{R}^d} |f(y)| \int_{0}^\infty \int_u^{\infty} T_{v}^V(x,y) (v-u)^{h-1} \, dv \, u^{t-1} \, du \, dy \\
        & = \int_{\mathbb{R}^d} |f(y)| \int_{0}^\infty \int_0^{v} (v-u)^{h-1} u^{t-1} \, du \, T_{v}^V(x,y) \, dv \, dy \\
        & = \frac{\Gamma(h)\Gamma(t)}{\Gamma(t+h)} \int_{\mathbb{R}^d} |f(y)| \int_{0}^\infty v^{h+t-1} T_{v}^V(x,y) \, dv \, dy \\
        & = \Gamma(h)\Gamma(t) \, \mathcal{L}_V^{-h-t}(|f|)(x) < \infty, \quad x \in \mathbb{R}^d.
    \end{split}
\end{equation*}

This justifies the interchange in the order of integration, obtaining 
\begin{equation}\label{eq: 3.2}
    \mathcal{L}_V^{-h}(\mathcal{L}_V^{-t} f)(x) = \mathcal{L}_V^{-(t+h)} f(x), \quad x \in \mathbb{R}^d.
\end{equation}

According to~\eqref{eq: 3.1} we get
\begin{equation*}
    \lim_{h \rightarrow 0^+} \frac{\mathcal{L}_V^{-(t+h)} f(x) - \mathcal{L}_V^{-t} f(x)}{h} = -\log(\mathcal{L}_V) (\mathcal{L}_V^{-t} f)(x), \quad x \in \mathbb{R}^d \text{ and } t \in (0, 1-\theta).
\end{equation*}
We are going to show that
\begin{equation}\label{eq: 3.3}
    \lim_{h \rightarrow 0^+} \frac{\mathcal{L}_V^{-(t-h)} f(x) - \mathcal{L}_V^{-t} f(x)}{-h} = -\log(\mathcal{L}_V) (\mathcal{L}_V^{-t} f)(x), \quad x \in \mathbb{R}^d \text{ and } t \in (0, 1-\theta).
\end{equation}

By using~\eqref{eq: 3.2} we get
\begin{equation*}
    \mathcal{L}_V^{-(t-h)} f(x) - \mathcal{L}_V^{-t} f(x) = \mathcal{L}_V^{-(t-h)} f(x) - \mathcal{L}_V^{-h} (\mathcal{L}_V^{-(t-h)} f)(x) = \frac{I - \mathcal{L}_V^{-h}}{h} \mathcal{L}_V^{-(t-h)} f(x),
\end{equation*}
for $x \in \mathbb{R}^d$ and $0 < h < t < 1-\theta$. 

We now prove that 
\begin{equation}\label{eq: 3.4}
\lim_{h \rightarrow 0^+} \mathcal{L}_V^{-(t-h)} f(x) = \mathcal{L}_V^{-t} f(x), \quad x \in \mathbb{R}^d \text{ and } 0 < t < 1-\theta.     
\end{equation}
Let $x \in \mathbb{R}^d$ and $0 < t < 1-\theta$. We decompose $\mathcal{L}_V^{-(t-h)} f(x)$ as follows:
\begin{equation*}
    \begin{split}
        \mathcal{L}_V^{-(t-h)} f(x) & = \frac{1}{\Gamma(t-h)} \left[
        \int_0^1 u^{t-h-1} \int_{\mathbb{R}^d} T_u^V(x,y) f(y) \, dy \, du + \int_1^\infty u^{t-h-1} \int_{\mathbb{R}^d} T_u^V(x,y) f(y) \, dy \, du \right] \\
        & =: I_1(h) + I_2(h), \quad 0 < h < t.
    \end{split}
\end{equation*}

Since $f \in \textup{Lip}_V^{\theta}$, by using~\eqref{eq: 1.2} we get, {for $0 < h < t/2$},
\begin{equation}\label{eq: 3.5}
    \begin{split}
        \int_0^1 u^{t-h-1} \int_{\mathbb{R}^d} T_u^V(x,y) |f(y)| \, dy \, du & \leq C \int_0^1 u^{t-h-1} \int_{\mathbb{R}^d} T_u(x-y) |f(y)| \, dy \, du \\
        & \leq C \int_0^1 u^{t/2-1} \int_{\mathbb{R}^d} T_u(x-y) |f(y)| \, dy \, du \\
        & \leq C \left( \int_{|x-y|<1} \frac{|f(y)|}{|x-y|^{d-t}} \, dy + \int_{\mathbb{R}^d} |f(y)| \, dy \right) \\
        & \leq C \left( \int_{|x-y|<1} \frac{|f(y)-f(x)| + |f(x)|}{|x-y|^{d-t}} \, dy + \int_{\mathbb{R}^d} |f(y)| \, dy \right) < \infty,
    \end{split}
\end{equation}
where $C$ does not depend on $h$. Then, the Dominated Convergence Theorem leads to
\begin{equation*}
    \lim_{h \rightarrow 0^+} I_1(h) = \frac{1}{\Gamma(t)} \int_0^1 u^{t-1} \int_{\mathbb{R}^d} T_u^V(x,y) f(y) \, dy \, du.
\end{equation*}

Since $f \in L^1(\mathbb{R}^d)$, by using again~\eqref{eq: 1.2} we obtain
\begin{equation}\label{eq: 3.6}
    \begin{split}
        \int_1^\infty u^{t-h-1} \int_{\mathbb{R}^d} T_u^V(x,y) |f(y)| \, dy \, du  & \leq C \int_1^\infty u^{t-h-1-d/2} \, du \int_{\mathbb{R}^d} |f(y)| \, dy \\
        & \leq C \int_1^\infty u^{t-1-d/2} \, du \int_{\mathbb{R}^d} |f(y)| \, dy < \infty, \quad 0 < h < t, 
    \end{split}
\end{equation}
where again $C$ does not depend on $h$. According to the Dominated Convergence Theorem we get
\begin{equation*}
    \lim_{h \rightarrow 0^+} I_2(h) = \frac{1}{\Gamma(t)} \int_{1}^{\infty} u^{t-1} \int_{\mathbb{R}^d} T_u^V(x,y) f(y) \, dy \, du,
\end{equation*}
and~\eqref{eq: 3.4} is proved.

By following carefully the proofs of Lemma 4.1 and Theorem 1.6 in~\cite{dLT}, we can deduce that if $0 < a < b < +\infty$ and $f \in \textup{Lip}^\theta_V$ with $\theta + b < 1$, there exists $C > 0$ such that
\begin{equation}\label{eq: 3.7}
    \|\mathcal{L}_V^{-s} f\|_{\textup{Lip}_V^{\theta+s}} \leq C \|f\|_{\textup{Lip}_V^{\theta}}, \quad s \in [a,b],
\end{equation}
where $C$ does not depend on $s \in [a,b]$. 

Let $x \in \mathbb{R}^d$ and $0 < t < 1-\theta$. By decomposing as in the proof of Theorem~\ref{thm: 1.2} and Corollary~\ref{coro: 2.1}, we consider
\begin{equation*}
    \begin{split}
        J_1(h) & = \int_{B(x,1)} \left( \mathcal{L}_V^{-(t-h)} f(x) - \mathcal{L}_V^{-(t-h)} f(y) \right) \int_0^\infty T_u^V(x,y) u^{h-1} \, du \, dy, \\
        J_2(h) & = \int_{\mathbb{R}^d \setminus B(x,1)} \mathcal{L}_V^{-(t-h)} f(y) \int_0^\infty T_u^V(x,y) u^{h-1} \, du \, dy, \\
        J_3(h) & = \left( - \int_{\mathbb{R}^d \setminus B(x,1)} \int_0^\infty T_u^V(x,y) u^{h-1} \, du \, dy 
        + \int_{0}^{\rho^2(x)} \int_{\mathbb{R}^d} \left( T_u^V(x,y) - T_u(x-y) \right) u^{h-1} \, dy \, du \right. \\
        &  + \left. \int_{\rho^2(x)}^\infty \int_{\mathbb{R}^d \setminus B(x,1)} T_u^V(x,y) u^{h-1} \, dy \, du
        + \int_{\rho^2(x)}^\infty \int_{B(x,1)} T_u^V(x,y) u^{h-1} \, dy \, du
        + \frac{\rho(x)^{2h}}{h} \right) \mathcal{L}_V^{-(t-h)} f(x),
    \end{split}
\end{equation*}
where $0 < h < t$.

Since $\mathcal{L}_V^{-(t-h)} f\in \textup{Lip}_V^{\theta+t+h}$ we deduce that there exists $C>0$ such that
\begin{equation*}
    |\mathcal{L}_V^{-(t-h)} f(x) - \mathcal{L}_V^{-(t-h)} f(y)| \leq C |x-y|^{\theta + t - h}, \quad y\in \RR^d \text{ and } 0<h<t/2,
\end{equation*}
and
\begin{equation*}
    |\mathcal{L}_V^{-(t-h)} f(y)| \leq C \rho(y)^{\theta+t-h}, \quad y\in \RR^d \text{ and } 0<h<t/2.
\end{equation*}

Since $0<t<1-\theta$ we have that, for every $y\in \RR^d$ and $0<h<t/2$
\begin{equation*}
    \rho(y)^{\theta+t-h} \leq \rho(y)^{\theta +t/2} + \rho(y)^{\theta+t} \leq \rho(y)^\theta + \rho(y),
\end{equation*}
and 
\begin{equation*}
    |x-y|^{\theta +t -h } \leq |x-y|^{\theta +t} + |x-y|^{\theta+t/2} \leq |x-y|^\theta + |x-y|.
\end{equation*}

By taking into acccount that
\begin{equation*}
    \int_{\RR^d} \frac{\rho(y)^\theta + \rho(y)}{(1+|y|)^{d+2h_0}} dy < \infty 
\end{equation*}
for some $h_0\in(0,1)$,
we can use the Dominated Convergence Theorem to get
\begin{equation*}
    \begin{split}
        \lim_{h\rightarrow 0^+} \frac{\mathcal{L}_V^{-h} - I}{-h}\mathcal{L}_V^{-(t-h)} f(x) & = \lim_{h\rightarrow 0^+} \frac{1}{h} \left( \frac{1}{\Gamma(h)} \left( J_1(h) + J_2(h) + J_3(h) \right) - \mathcal{L}_V^{-t} f(x)\right)
        \\ & = -(\log \mathcal{L}_V) \mathcal{L}_V^{-t} f(x).
    \end{split}
\end{equation*}
Thus,~\eqref{eq: 3.3} is proved.

By combining~\eqref{eq: 3.1} and~\eqref{eq: 3.3} we conclude that $u(x,t) = \mathcal{L}_V^{-t} f(x)$, $x\in \RR^d$ and $t>0$, is in $D(\log \mathcal{L}_V)$ and 
\begin{equation*}
    \frac{\partial}{\partial t} u(x,t) = -\log(\mathcal{L}_V)) u(x,y),\quad x\in \RR^d \text{ and } t>0.
\end{equation*}

Since $\mathcal{L}_V^{-t}(f)\in \textup{Lip}_V^{\theta +t}$, $0<t<1-\theta$, the function $u(\cdot, t)$ is continuous in $\RR^d$ for every $t\in (0,1-\theta)$. 

In the first part of the proof we have proved that the function $u(x,\cdot)$ is continuous in $t\in(0,1-\theta)$, for every $x\in \RR^d$. We are going to see that the function $u$ is continuous in $\RR^d\times (0,1-\theta)$. Let $(x_0,t_0)\in \RR^d\times (0,1-\theta)$.  Suppose that, for every $k\in \NN$, $(x_k,t_k)\in \RR^d\times (0,1-\theta)$ and that $(x_k,t_k)\rightarrow (x_0,t_0)$, as $k\rightarrow \infty$. We can write
\begin{equation*}
    |u(x_k,t_k) - u(x_0,t_0)| \leq 
    | \mathcal{L}_V^{-t_k}(f)(x_k) - \mathcal{L}_V^{-t_k}(f)(x_0)| + |\mathcal{L}_V^{-t_k}(f)(x_0) - \mathcal{L}_V^{-t_0}(f)(x_0)|, \quad k\in\NN.
\end{equation*}

We chose $\delta_0>0$ such that $[t_0 - \delta_0,t_0+ \delta_0]\subseteq (0,1-\theta)$. There exists $k_0\in \NN$ such that $t_k \in [t_0-\delta_0, t_0+\delta_0]$ and $|x_k-x_0|<1$, for $k\geq k_0$. According to~\eqref{eq: 3.7}, there exists $C>0$ such that
\begin{equation}\label{eq: 3.8}
    |\mathcal{L}_V^{-t_k}(f)(x_k) - \mathcal{L}_V^{-t_{k}} (f)(x_0) | \leq C |x_k - x_0|^{\theta + t_k} \leq C |x_k - x_0|^{\theta+t_0-\delta_0}, \quad k\geq k_0.
\end{equation}

On the other hand, using~\eqref{eq: 1.2} we obtain 
\begin{equation*}
    \begin{split}
        |\mathcal{L}_V^{-t_k}(f)(x_0) - \mathcal{L}_V^{-t_0}(f)(x_0)| & \leq \int_{0}^\infty \left| \frac{u^{t_k}-1}{\Gamma(t_k)} - \frac{u^{t_0-1}}{\Gamma(t_0)}\right| \int_{\RR^d} T_u^V(x_0,y)|f(y)| dy du
        \\ & \leq \int_{0}^\infty \left| \frac{u^{t_k}-1}{\Gamma(t_k)} - \frac{u^{t_0-1}}{\Gamma(t_0)}\right| \int_{\RR^d} \frac{e^{\frac{-c|x_0-y|^2}{u}}}{u^{d/2}}|f(y)| dy du, \quad k\in \NN.
    \end{split}
\end{equation*}

Since $t_0+\delta_0-d/2<0$, we have that 
\begin{equation}\label{eq: 3.9}
\begin{split}
    \int_1^\infty \left| \frac{u^{t_k - 1}}{\Gamma(t_k)} - \frac{u^{t_0 - 1}}{\Gamma(t_0)} \right| \int_{\mathbb{R}^d} \frac{e^{-c|x_0 -y|^2/u}}{u^{d/2}} |f(y)| dy du
    &  \leq C \int_1^{\infty} \left( \frac{u^{t_k - 1}}{\Gamma(t_k)} + \frac{u^{t_0 - 1}}{\Gamma(t_0)} \right) \frac{1}{u^{d/2}} \int_{\RR^d} |f(y)| dy du
    \\ & \leq C \int_1^\infty u^{t_0+\delta_0-1-d/2} \int_{\RR^d} |f(y)| dy du <\infty, 
\end{split}
\end{equation}
for $k\geq k_0$. Since $\rho(y)\simeq \rho(x)$ when $|x_0 - y| < \rho(x)$ we get

\begin{equation}\label{eq: 3.10}
    \begin{split}
\int_0^1 & \left| \frac{u^{t_k-1}}{\Gamma(t_k)} - \frac{u^{t_0-1}}{\Gamma(t_0)}\right|  \int_{\RR^d} \frac{e^{-c|x_0-y|^2/u}}{u^{d/2}} |f(y)| dy du
 \\ & C\leq \int_0^1 u^{t_0 - \delta_0 -1} \left( \int_{|x_0 - y|<\rho(x_0)}  \frac{e^{-c|x_0-y|^2/u}}{u^{d/2}} |f(y)| dy   + \int_{|x_0 - y|\geq\rho(x_0)}  \frac{e^{-c|x_0-y|^2/u}}{u^{d/2}} |f(y)| dy 
\right) du
\\ & C\leq \int_0^1 u^{t_0 - \delta_0 -1} \left( \int_{|x_0 - y|<\rho(x_0)}  \frac{e^{-c|x_0-y|^2/u}}{u^{d/2}} \rho(y)^\theta dy   + \int_{|x_0 - y|\geq\rho(x_0)}  \frac{|f(y)|}{|x_0 - y|^d} dy 
\right) du
\\ & C\leq \int_0^1 u^{t_0 - \delta_0 -1} \left( \int_{\RR^d}  |f(y)| dy \rho(x_0)^{-d} + \rho(x_0)^\theta  
\right) du <\infty,
\end{split}
\end{equation}
for $k<k_0$. Nothe that the constants $C$ in~\eqref{eq: 3.9} and~\eqref{eq: 3.10} does not depend on $k$. 
By combining~\eqref{eq: 3.9} and~\eqref{eq: 3.10} and by using the Dominated Convergence Theorem we obtain
\begin{equation}\label{eq: 3.11}
    \lim_{k\rightarrow \infty} \mathcal{L}_V^{-t_k} (f)(x_0) = \mathcal{L}_V^{-t_0} (f)(x_0).
\end{equation}

By putting together~\eqref{eq: 3.8} and~\eqref{eq: 3.11} we conclude that

\begin{equation}
    \lim_{k\rightarrow \infty} \mathcal{L}_V^{-t_k} (f)(x_k) = \mathcal{L}_V^{-t_0} (f)(x_0).
\end{equation}
Thus, we prove that $u$ is continuous in $\RR^d \times (0,1-\theta)$. 

We now see that $\frac{\partial u}{\partial t}$ is continuous in $\RR^d \times (0,1-\theta)$. We have that

\begin{equation*}
    \frac{\partial u (x,t)}{\partial t} = \frac{1}{\Gamma(t)} \int_0^\infty u^{t-1} \log u \, T_u^V(f)(x) du {-} \frac{\Gamma'(t)}{\Gamma(t)} u(x,t), \qquad x\in\RR^d \text{ and } t\in (0,1-\delta). 
\end{equation*}

Differentiation under the integral sign is justified because we have that
\begin{equation*}
    \begin{split}
        \int_0^\infty u^{t-1} |\log u| \int_{\RR^d} T_u^V(x,y) |f(y)| dy du & \leq C \left( \int_0^1 u^{t-1} |\log u|  \int_{|x-y|<\rho(x)} \frac{e^{-c|x-y|^2/u}}{u^{d/2}} |f(y)| dy du \right.
        \\ & \qquad + \int_0^1 u^{t-1} |\log u|  \int_{|x-y|\geq\rho(x)} \frac{e^{-c|x-y|^2/u}}{u^{d/2}} |f(y)| dy du
        \\ & \qquad \left. + 
        \int_1^\infty u^{t-1} |\log u|  \int_{\RR^d} \frac{e^{-c|x-y|^2/u}}{u^{d/2}} |f(y)| dy du\right)
        \\ & \leq C  \left( \rho(x)^{-\theta}\int_0^1 u^{t-1} |\log u|  du  \right.
        \\ & \qquad +  \rho(x)^{-d} \int_0^1 u^{t-1} |\log u|   \int_{|x-y|\geq \rho(x)} {|f(y)| dy du }
        \\ & \qquad \left. + 
        \int_1^\infty u^{t-1-d/2} |\log u|  \int_{\RR^d}  |f(y)| dy du\right) <\infty,
    \end{split}
\end{equation*}
for $x\in \RR^d$ and $t_0\in(0,1-\theta)$. We choose a sequence $\{(x_k,t_k)\}_{k\in \NN}\subseteq \RR^d \times (0,1-\theta)$ such that  $\displaystyle \lim_{k\rightarrow \infty} (x_k,t_k) = (x_0,t_0)$. We can write
\begin{equation*}
\begin{split}
    \int_0^\infty u^{t_k-1} \log u T_u^V(f)(x_k)du & = \int_1^\infty u^{t_k-1} \log u \int_{\RR^d} T_u^V (x_k,y) f(y) dy du 
    \\ & \qquad +
    \int_{0}^1 u^{t_k -1} \log u \int_{|y-x_0|\geq \rho(x_0)} T_u^V (x_k,y) f(y) dy du
    \\ & \qquad + \int_{0}^1 u^{t_k -1} \log u \int_{|y-x_0|< \rho(x_0)} T_u^V (x_k,y) f(y) dy du
    \\ & = J_1(k) + J_2(k) + J_3(k), \qquad  k\in \NN.
\end{split}
\end{equation*}

We choose $\delta_0>0$ such that $[t_0 - \delta_0, t_0 + \delta_0]\subseteq (0,1-\theta)$. There exists $k_0\in \NN$ such that $t_k \in (t_0 - \delta_0, t_0 + \delta_0)$, $k\geq k_0$. 

We now apply~\eqref{eq: 1.2} to get
\begin{equation*}
    \begin{split}
        u^{t_k-1} |\log u | T_u^V(x_k,y) |f(y)| \leq C u^{t_0+\delta_0 - 1} |\log u | \frac{e^{-c|y-x_k|^2/u}}{u^{d/2}} |f(y)| \leq C |\log u | u^{{t_0 + \delta_0 -d/2- 1}} |f(y)|, 
    \end{split}
\end{equation*}
for $y\in \RR^d$, $u\in(1,\infty)$ and $k\geq k_0$, being
\[
\int_1^\infty \int_{\RR^d} u^{t_0+\delta_0 - 1 - d/2} |f(y)|dy du <\infty. 
\]
Then
\[
\lim_{k\rightarrow\infty} J_1(k) = \int_1^\infty u^{t_0-1} \log u \int_{\RR^d} T_u^V(x_0,y) f(y) dy.
\]

On the other hand, there exists $k_1\geq k_0$, $k_1\in \NN$ such that $|x_0 - x_k| < \rho(x_0)/2$. We obtain
\begin{equation*}
    \begin{split}
        u^{t_k-1}|\log u | T_u^V(x_k,y) |f(y)|  & \leq C u^{t_0 - \delta_0 - 1} |\log u| \frac{e^{-c|y-x_k|^2/u}}{u^{d/2}} |f(y)|
        \\ & \leq  C u^{t_0 - \delta_0 - 1} |\log u| \frac{|f(y)|}{|y-x^k|^d}
        \\ & \leq C u^{t_0 - \delta_0 - 1} |\log u| |f(y)| \rho(x_0)^{-d}, 
    \end{split}
\end{equation*}
for $y\in \RR^d$, $|y-x_0|\geq \rho(x_0)$,  $0<t<1$ and $k\geq k_1$, because $|y-x_k| \geq |y-x_0| - |x_0 - x_k| \geq \rho(x_0) - \rho(x_0)/2 = \rho(x_0)/2$, when $|y-x_0|\geq \rho(x_0)$ and $k\geq k_1$. 

Since 
\[\int_0^1 \int_{\RR^d} u^{t_0-\delta_0 - 1} |\log u| |f(y)|dy du <\infty,\]
it follows that
\[\lim_{k\rightarrow \infty} J_2(k) = \int_0^1 u^{t_0-1} \log u \int_{{|x-y|\geq \rho(x_0)}} T_u^V(x_0,y) f(y) dy.\]

We also have that, since $f\in \Lip_V^\theta$, we get
\begin{equation*}
    \begin{split}
        u^{t_k-1} |\log u | \int_{|y-x_0|<\rho(x_0)} T_t(x_k,y) |f(y)| dy & \leq C u^{t_0 - {\delta_0} - 1} |\log u| \int_{|y-x_0|<\rho(x_0)} \frac{e^{-c|x_k-y|^2/u|}}{u^{d/2}} |f(y)| dy 
        \\ & \leq C u^{t_0 - {\delta_0} - 1} |\log u| \int_{|y-x_0|<\rho(x_0)} \frac{e^{-c|x_k-y|^2/u|}}{u^{d/2}} \rho(y)^\theta dy
        \\ & \leq C \rho(x_0)^\theta  u^{t_0 - {\delta_0} - 1} |\log u| \int_{|y-x_0|<\rho(x_0)} \frac{e^{-c|x_k-y|^2/u|}}{u^{d/2}} dy,
    \end{split}
\end{equation*}
for $u\in (0,1)$ and $k\geq k_0$, and 
\[ T_u^V(x_k,y) |f(y)|  \leq C \frac{e^{-c|x_k-y|^2/2}}{u^{d/2}}|f(y)| \leq C\frac{|f(y)|}{u^{d/2}},\]
for $u>0$, $y\in \RR^d$ and $k\in \NN$.
Then
\[\lim_{k\rightarrow \infty } \int_{|y-x_0|<\rho(x_0)} T_u^V(x_k,y) f(y) dy = \int_{|y-x_0|<\rho(x_0)} T_u^V(x_0,y) f(y) dy, \qquad u>0,\]
and
\[
\begin{split}
\lim_{k\rightarrow \infty } J_3(k) & = \int_0^1 u^{t_0-1} \log u \lim_{k\rightarrow \infty }  \int_{|y-x_0|<\rho(x_0)} T_u^V(x_k,y) f(y) dy du
\\ & = \int_0^1 u^{t_0-1} \log u  \int_{|y-x_0|<\rho(x_0)} T_u^V(x_0,y) f(y) dy du
\end{split}
\]

Then, we prove that 
\[\lim_{k\rightarrow\infty} \int_0^\infty u^{t_k-1} \log u   T_u^V(f)(x_k) du = \int_0^\infty u^{t_0-1} \log u   T_u^V(f)(x_0) du.\]
We conclude that $\frac{\partial u}{\partial t}$, and hence $(\log \mathcal{L}_V)u$ is continuous in $\RR^d\times (0,1-\theta)$.

Our next objective is to see that
\begin{equation}\label{eq: 3.12}
\lim_{t\rightarrow 0^+} u(x,t) = f(x), \qquad x\in \RR^d.    
\end{equation}
Let $x\in \RR^d$. We choose $R>0$ such that $R>|x|+1$. We decompose $u$ as follows.
\begin{equation*}
    \begin{split}
        u(x,t) & = \frac{1}{\Gamma(t)} \int_0^\infty u^{t-1} \int_{B(x,R)} T_u^V(x,y) f(y) dy du + \frac{1}{\Gamma(t)} \int_0^\infty u^{t-1} \int_{\RR^d\setminus B(x,R)} T_u^V(x,y) f(y) dy du
        \\ & = J_1(t) + J_2(t), \qquad 0<t<1-\theta.
    \end{split}
\end{equation*}

First we prove that $J_1(t) \rightarrow f(x)$ as $t\rightarrow 0^+$. We can write
\begin{equation*}
    \begin{split}
        J_1(t) & = \frac{1}{\Gamma(t)} \int_0^{\rho^2(x)} u^{t-1} \int_{B(x,R)} \left( T_u^V(x,y) - T_u(x-y)\right) f(y) dy
        \\ & \qquad + \frac{1}{\Gamma(t)} \int_{\rho^2(x)}^\infty u^{t-1} \int_{B(x,R)}  T_u^V(x,y)  f(y) dy
        \\ & \qquad - \frac{1}{\Gamma(t)} \int_{\rho^2(x)}^\infty u^{t-1} \int_{B(x,R)}  T_u(x-y)  f(y) dy
        \\ & \qquad + \frac{1}{\Gamma(t)} \int_{0}^\infty u^{t-1} \int_{B(x,R)}  T_u(x-y)  f(y) dy
        \\ & = J_{1,1}(t) +J_{1,2}(t) + J_{1,3}(t) + J_{1,4}(t), \qquad 0<t<1-\theta.
    \end{split}
\end{equation*}
According to~\eqref{eq: 1.5} we have that
\[ |J_{1,1}(t)| \leq C \frac{1}{\Gamma(t)}
\int_0^{\rho^2(x)} u^{t-1} \int_{B(x,R)} \left(\frac{\sqrt{u}}{\rho(x)} \right)^\delta \omega_k(x-y) |f(y)| dy du, \qquad 0<t<1-\theta.\]

By using~\eqref{eq: rox_vs_roy} and a compactness argument as we did in the proof of Theorem~\ref{thm: 1.2}\ref{item: teo 1.2 - b}, we can find $C>1$ such that
\[\frac{1}{C} \leq \rho(y)\leq C, \qquad y\in B(x,R).\]

Since $f\in \Lip_V^\theta$ we deduce that
\[
|J_{1,1}(t)| \leq \frac{C}{\Gamma(t)} \int_0^{\rho^2(x)} u^{t-1+\delta/2} \int_{B(x,R)} \omega_u(x-y) dy du \leq \frac{C}{\Gamma(t)} \frac{\rho(x)^{2(t+\delta/2)}}{t+\delta/2}, \qquad 0<t<1-\theta.
\]
Then, $\displaystyle\lim_{t\rightarrow 0^+} J_{1,1}(t) = 0$. 

From \eqref{eq: 1.2} we get
\begin{align*}
    \left|J_{1,2}(t)\right|+ \left|J_{1,3}(t)\right| &\leq  C\frac{1}{\Gamma(t)}\int_{\rho^2(x)}^\infty u^{t-1}\int_{B(x,R}\frac{e^{-c\frac{|x-y|^2}{u}}}{u^{\frac{d}{2}}}|f(y)|dydu\\
    & \leq  C\frac{1}{\Gamma(t)} \int_{\rho^2(x)}^\infty u^{t- \frac{\delta}{2}-1}du \int_{\mathbb R^d} |f(y)|dy \\
    & \leq  C \frac{1}{\Gamma(t)}\frac{\rho(x)^{2t-n}}{\frac{d}{2}-t},\quad 0<t<1-\theta 
\end{align*}
It follows that $\displaystyle\lim_{t\rightarrow 0}J_{1,2}(t)+J_{1,3}(t))=0$.

By interchanging the order of integration we obtain
\begin{align*}
    J_{1,4}(t) &=\frac{1}{\Gamma(t)}\int_{B(x,R)}f(y)\int_0^\infty\frac{e^{-\frac{|x-y|^2}{4u}}}{(4\pi u)^{\frac{d}{2}}}u^{t-1}du dy\\
    &= \frac{4^{-t}\Gamma\left(\frac{d-2t}{2}\right)}{\pi^{\frac{d}{2}\Gamma(t)}}\int_{B(x,R)}\frac{f(y)}{|x-y|^{d-2t}}dy,\quad 0<t<1-\theta.
\end{align*}
As it was proved in the proof of \cite[Proposition 2.3]{ChV1} we have that
\[
\lim_{t \rightarrow 0} J_{1,4}(t) = f(x).
\]
On the other hand by using again \eqref{eq: 1.2} we get
\begin{align*}
    J_{2}(t) &=\frac{1}{\Gamma(t)}\int_0^\infty u^{t-1}\int_{\mathbb R^d\setminus B(x,R)}\frac{e^{-\frac{|x-y|^2}{4u}}}{(4\pi u)^{\frac{n}{2}}}|f(y)|dy du\\ 
    &= \frac{4^{-t}\Gamma\left(\frac{d-2t}{2}\right)}{\pi^{\frac{d}{2}\Gamma(t)}}\int_{\mathbb R^d \setminus B(x,R)}\frac{|f(y)|}{|x-y|^{n-2t}}dy,\quad 0<t<1-\theta.
\end{align*}

By using now H\"older and Gagliardo-Nirenberg inequalities as in the proof of \cite[Proposition 2.3]{ChV1} we conclude that
\[
\lim_{t \rightarrow 0} J_2(t)=0.
\]

\bibliographystyle{acm}

\end{document}